\documentclass[12pt]{amsart}
\usepackage[foot]{amsaddr}

\usepackage{amsmath, amssymb, amsthm}
\usepackage{dsfont}
\usepackage[T1]{fontenc}    % use 8-bit T1 fonts
\usepackage{forloop}
\usepackage{fullpage}
\usepackage{graphicx}
\usepackage[acronym, symbols]{glossaries}
\usepackage{hyperref}
\usepackage[utf8]{inputenc}
\usepackage{lmodern}
\usepackage{mathtools}
\usepackage{natbib}
\usepackage{nicefrac}
\usepackage{paralist}
\usepackage{pgfmath}
\usepackage{subcaption}
\usepackage{tabularx}
\usepackage{thmtools}
\usepackage{thm-restate}
\let\todo\relax
\usepackage{todonotes}
\usepackage{tikz}
\usepackage{tikz-cd}
\usepackage{url}
\usepackage{verbatim}
\usepackage{wrapfig}
\usepackage{xargs}
\usepackage{xspace}

\usepackage{cleveref}
\usepackage{enumitem}
\usetikzlibrary{positioning, fit, shapes.geometric,intersections}
\usepackage[noend,ruled,algo2e]{algorithm2e}
\graphicspath{ {images/} }

\newtheorem{theorem}{Theorem}[section]
\newtheorem{proposition}[theorem]{Proposition}
\newtheorem{corollary}[theorem]{Corollary}

\newtheorem{lemma}[theorem]{Lemma}

\newtheorem{problem}[theorem]{Problem}

\theoremstyle{definition}
\newtheorem{definition}[theorem]{Definition}
\newtheorem{example}[theorem]{Example}
\newtheorem{remark}[theorem]{Remark}
\newtheorem{construction}[theorem]{Construction}
\crefname{construction}{construction}{constructions} 
\Crefname{construction}{Construction}{Constructions} 
\crefname{algoline}{Line}{Lines}

\newcommandx{\moritz}[3][1=, 2=]{\todo[color=teal!50, #1, caption={#2}]{Moritz: #3}}
\newcommand{\complex}{\mathcal{C}}
\newcommand{\hyperplane}[1]{H_{#1}}
\newcommand{\halfspace}[2]{H^{#1}_{#2}}

\newcommand{\defcone}{\mathcal{DC}}
\newcommand{\defconemod}{\overline{\mathcal{DC}}}
\newcommand{\polytopes}{\mathcal{P}}
\newcommand{\vpolytopes}{\mathcal{VP}}
\newcommand{\aff}{\operatorname{aff}}
\newcommand{\Aff}{\operatorname{Aff}}
\newcommand{\argmax}{\operatorname{argmax}}
\newcommand{\B}{\mathcal{B}}
\newcommand{\cone}{\operatorname{cone}}
\newcommand{\conv}{\operatorname{conv}}
\newcommand{\cpwl}{\operatorname{CPWL}}

\newcommand{\F}{\mathcal{F}}

\newcommand{\ind}{\mathds{1}}

\newcommand{\Newt}{\operatorname{Newt}}
\newcommand{\PP}{\mathcal{P}}

\newcommand{\recc}{\operatorname{recc}}

\newcommand{\spn}{\operatorname{span}}

\newcommand{\str}{\operatorname{star}}
\newcommand{\supp}{\operatorname{supp}}
\newcommand{\FB}[1]{\mathcal{V}_{\mathcal{B}_{#1}}}
\newcommand{\Sf}[1]{\mathcal{F}_{#1}}
\newcommand{\inner}[1]{\langle #1 \rangle}

\newcommand{\ma}{\begin{pmatrix} }
\newcommand{\trix}{\end{pmatrix} }
\newcommand{\sma}{ \left( \begin{smallmatrix} }
\newcommand{\strix}{ \end{smallmatrix} \right) }

\newcommand{\decomp}[2]{\mathcal{D}_{#1}(#2)}
\newcommand{\VC}{\mathcal{V}_\complex}

\newcommand{\WC}{\mathcal{W}_\complex}
\newcommand{\VCP}{\overline{\mathcal{V}}_\complex^+}

\newcommand{\WCP}{\mathcal{W}_\complex^+}

\newcommand{\VCmod}{\overline{\mathcal{V}}_\complex}

\newcommand{\V}[1]{\mathcal{V}_{#1}}
\newcommand{\Vmod}[1]{\overline{\mathcal{V}}_{#1}}

\newcommand{\VPmod}[1]{\overline{\mathcal{V}}^+_{#1}}
\newcommand{\WP}[1]{\mathcal{W}^+_{#1}}

\newcommand{\submod}{\mathcal{SM}}
\newcommand{\modular}{\mathcal{MOD}}

\newcommand{\GPmod}{\overline{\mathcal{GP}}}

\newcommand{\R}{\mathbb{R}}

\tikzset{connection/.style={ -{Stealth[scale=1.5]}}}
\tikzset{neuron/.style={thick, circle, draw, inner sep = 0ex, minimum width = 4.5ex}}

\definecolor{concave}{cmyk}{.35,0,.99,0}
\definecolor{convex}{RGB}{0,116,125}
\definecolor{mpblue}{RGB}{0,116,125}
\definecolor{mpred}{RGB}{139,0,0}
\definecolor{mppetrol}{cmyk}{1,.26,.45,.16}
\definecolor{mpgreen}{cmyk}{.35,0,.99,0}
\definecolor{mpgray}{gray}{.261}%.739

\title{Decomposition Polyhedra of Piecewise\\Linear Functions}
\thanks{A preliminary conference version (with unreviewed appendix) appeared as a spotlight at ICLR 2025 \citep{brandenburg2025decomposition}.}
\author{Marie-Charlotte Brandenburg$^{1}$}
\address{$^1$Ruhr-Universit\"at Bochum} 
\email{marie-charlotte.brandenburg@rub.de}
\author{Moritz Grillo$^2$}
\address{$^2$Max Planck Institute MiS}
\email{moritz.grillo@mpg.mis.de}

\author{Christoph Hertrich$^3$}
\address{$^3$University of Technology Nuremberg} 
\email{christoph.hertrich@utn.de}

\begin{document}

\begin{abstract}
    In this paper we contribute to the frequently studied question of how to decompose a continuous piecewise linear (CPWL) function into a difference of two convex CPWL functions. Every CPWL function has infinitely many such decompositions, but for applications in optimization and neural network theory, it is crucial to find decompositions with as few linear pieces as possible. This is a highly challenging problem, as we further demonstrate by proving that a recently proposed approach by \citet{tran2024minimal} does not extend beyond two dimensions. To make the problem more tractable, we propose to fix an underlying polyhedral complex determining the possible locus of nonlinearity. Under this assumption, we prove that the set of decompositions forms a polyhedron that arises as intersection of two translated cones. We prove that irreducible decompositions correspond to the bounded faces of this polyhedron and minimal solutions must be vertices. We then identify cases with a unique minimal decomposition, and illustrate how our insights have consequences in the theory of submodular functions. Finally, we improve upon previous constructions of neural networks for a given convex CPWL function and apply our framework to obtain results in the nonconvex case.
\end{abstract}
\maketitle

\section{Introduction}
Continuous piecewise linear (CPWL) functions play a crucial role in optimization and machine learning. While they have traditionally been used to describe problems in geometry, discrete and submodular optimization, or statistical regression, they recently gained significant interest as functions represented by neural networks with rectified linear unit (ReLU) activations~\citep{arora2018understanding}. Extensive research has been put into understanding which neural network architectures are capable of representing which CPWL functions~\citep{bakaev2025better,chen2022improved,haase2023lower,hertrich2021towards}. A major source of complexity in all the aforementioned fields is nonconvexity. Indeed, not only are nonconvex optimization problems generally much harder to solve than convex ones, but also for neural networks, nonconvexities are usually responsible for making the obtained representations complicated.

It is a well-known folklore fact that every (potentially nonconvex) CPWL function $f\colon \R^n\to\R$ can be written as the difference $f=g-h$ of two \emph{convex} CPWL functions \citep{melzer1986expressibility,KS87}. Consequently, a natural idea to circumvent the challenges induced by nonconvexity is to use such a decomposition $f=g-h$ and solve the desired problem separately for $g$ and $h$. This is the underlying idea of many successful optimization routines, known as DC programming (see survey by \citet{le2018dc}), and also occurs in the analysis of neural networks \citep{zhang2018tropical,hertrich2024neural}. However, the crucial question arising from this strategy is: how much more complex are $g$ and $h$ compared to~$f$? A well-established measure for the complexity of a CPWL function is the number of its linear pieces. Therefore, the main question we study in this article is the following.

\begin{problem}\label{problem}
	How to decompose a CPWL function $f$ into a difference $f=g-h$ of two convex CPWL functions with as few pieces as possible?
\end{problem}

There exist many ways in the literature to obtain decompositions into two convex functions, as we discuss later, but none of them guarantees minimality or at least a useful bound on the number of pieces of $g$ and $h$ depending on those of $f$. In fact, no finite procedure is known that guarantees to find a minimal decomposition, despite a recent attempt by \citet{tran2024minimal}.

\subsection{Our Contributions}

In this article, we propose a novel perspective on \Cref{problem} making use of polyhedral geometry and prove a number of structural results. We then apply our approach to existing decompositions in the literature, as well as to the theory of submodular functions and to the construction of neural networks representing a given CPWL function, serving as an additional motivation. Our detailed contributions are outlined as follows.

\subsubsection*{Decomposition Polyhedra.}

After setting the preliminaries in \Cref{sec:prelims} and introducing different ways to represent and study CPWL functions in \Cref{sec:different_perspectives}, \Cref{sec:compatible} presents our new polyhedral approach to \Cref{problem}. Instead of aiming for a globally optimal decomposition, we propose to restrict to solutions that are \emph{compatible with a given regular polyhedral complex}~$\complex$. In short, this means fixing where the functions $g$ and $h$ may have breakpoints, that is, points where they are not locally linear. We prove that the set of solutions to decompose $f$ in a way that is compatible with $\complex$ is a polyhedron $\decomp{\complex}{f}$ that arises as the intersection of two shifted polyhedral cones (\Cref{prop:shifted-cones}). We call this the \emph{decomposition polyhedron} of $f$ with respect to $\complex$. We prove several structural properties of $\decomp{\complex}{f}$. Among them, we show that the bounded faces of $\decomp{\complex}{f}$ are exactly those that cannot easily be simplified by subtracting a convex function (\Cref{th:decomp-reduced}), and we show that a minimal solution must be a vertex of $\decomp{\complex}{f}$ (\Cref{th:minimal-vertex}). The latter implies a finite procedure to find a minimal decomposition among those that are compatible with $\complex$, by simply enumerating the (potentially many) vertices of $\decomp{\complex}{f}$. It also implies that, if there is a unique vertex, then there is a unique minimal decomposition. We demonstrate that this is indeed the case for important CPWL functions, e.g., the median function, or those computed by a 1-hidden-layer ReLU network.

\subsubsection*{Translation to Virtual polytopes}
In \Cref{sec:virtual}, we demonstrate that our framework naturally translates to the setting of \emph{virtual polytopes} \citep{panina2015virtual}. We explain how decomposing a positively homogeneous CPWL function corresponds to finding representatives of a virtual polytope as a Minkowski difference of two conventional polytopes. Our main results translate directly, showing that minimal representatives, those with the fewest vertices, must correspond to vertices of an analogous ``representative polyhedron''. 
\subsubsection*{Existing Decompositions.} Afterwards, in \Cref{sec:existing}, we put our investigations into a broader context within the existing literature. We compare our minimality conditions with existing methods to construct decompositions. Notably, in this context, we refute a conjecture by \citet{tran2024minimal}, who provide an optimal construction method in dimension 2 and suggest that it might generalize to higher dimensions. However, we show that the construction does not generalize in the suggested way.

\subsubsection*{Applications to Submodular Functions.} 
In \Cref{sec:submodular}, we show that our framework entails the setup of set functions which are decomposed into differences of submodular set functions. Representing a set function as such a difference is a popular approach to solve optimization problems similarly to DC programming, see~\citet{Narasimhan2005,Iyer2012,halabi23}. We apply our results from \Cref{sec:compatible} to obtain analogous structural insights about (submodular) set functions (\Cref{cor:submod}).
\subsubsection*{Application to Neural Network Constructions.}
Finally, in \Cref{sec:NN}, we study the problem of constructing neural networks representing a given CPWL function. For convex CPWL functions, we blend two incomparable previous constructions by \citet{hertrich2021towards} and \citet{chen2022improved} to let the user freely choose a trade-off between depth and width of the constructed networks. We then apply the results of this paper to extend this to the nonconvex case by first decomposing the input function as a difference of two convex ones.

\subsubsection*{Limitations.}

We emphasize that the focus of our paper is fundamental research by building a theoretical foundation to tackle \Cref{problem} and connecting it with other fields. As such, our paper does not imply any direct improvement for a practical task, but it might prove helpful for that in the future. In particular, it is beyond the scope of our paper to provide any implementation of a (heuristic or exact) method to decompose a CPWL function into a difference of two convex ones. We consider it an exciting avenue for future research to do so, and to apply it to DC programming, discrete optimization, or neural networks.

On the theoretical side, the approach of fixing an underlying compatible polyhedral complex imposes some restriction on the set of possible solutions and can therefore be seen as a limitation. However, we think that this assumption is well justified by the structural properties this assumption allows us to infer and by the examples we demonstrate to fit into the framework. Even with this assumption, the problem remains very challenging.

\subsection{Further Related Work}

Explicit constructions to decompose CPWL functions as differences of convex CPWL functions can be found in several articles, such as \citet{KS87,Zalgaller2000,1333237,schlüter2021novel}. This was initiated in the $1$-dimensional case by \citet{butner70_somerepresentationtheorems}, and already laid out for positively homogeneous functions in general dimensions by \citet{melzer1986expressibility}.
Typically, such decompositions are based on certain representations of CPWL functions, which have been constructed, e.g., in \citet{tarela1999region,Ovchinnikov2002,wang2005generalization}; see also \citet{koutschan2023representing,koutschan2024representing} for a fresh perspective.
These representations also help to understand the representative capabilities of neural networks, see \citet{arora2018understanding,hertrich2021towards,chen2022improved}. 

Recently, a minimal decomposition for the $2$-dimensional case was given by \citet{tran2024minimal}. They use a duality between CPWL functions and  polyhedral geometry, based on the ``balancing condition'' from tropical geometry. This condition has already been studied by \citet{McMullen1996} in terms of weight spaces of polytopes. Generally, methods from tropical geometry have been successfully used to understand the geometry of neural networks \citep{zhang2018tropical,brandenburg2024the,montufar2022sharp}. In this context, a central open question is the required depth to exactly represent CPWL functions on $\R^n$ \citep{hertrich2021towards,haase2023lower,grillo2025depth,bakaev2025better,bakaev2025depth}.

Submodular functions are sometimes called the discrete analogue of convex functions, and optimizing over them is a widely studied problem, which is also relevant for machine learning. A submodular function can be minimized in polynomial time \citep{grotschel1981ellipsoid}. In analogy to DC programming, this sparked the idea of minimizing a general set function by representing it as a difference of two submodular ones \citep{Narasimhan2005,Iyer2012,halabi23}. Related decompositions were recently studied by \citet{berczi2024monotonicdecompositionssubmodularset}. In polyhedral theory, such a decomposition is equivalent to Minkowski differences of generalized permutahedra \citep{Ardila2009,jochemko22_generalizedpermutahedraminkowski}.

Another closely related stream of work is concerned with the (exact and approximate) representative capabilities of neural networks, starting with universal approximation theorems~\citep{cybenko1989approximation}, and specializing to ReLU networks, their number of pieces, as well as depth-width-tradeoffs~\citep{telgarsky2016benefits,eldan2016power,arora2018understanding,ergen2024topological}. In addition, \citet{hertrich2023provably,hertrich2024relu,hertrich2024neural} provide neural network constructions and bounds for CPWL functions related to combinatorial optimization problems. Geometric insights have also proven to be useful to understand the computational complexity of tasks related to neural networks, like training and verification \citep{froese2022computational,bertschinger2024training,froese2024training,froese2025complexity,stargalla2025computational}. Recently, \citet{safran2024many} give an explicit construction of how to efficiently approximate the maximum function with ReLU networks.

\section{Preliminaries}
\label{sec:prelims}
In this section we introduce the necessary preliminaries on polyhedral geometry and CPWL functions; see \citet{ziegler_lecturespolytopes} for more details. 
\subsection{Polyhedra and Polyhedral Complexes}
For a vector $a \in \R^{d}$ and a scalar $b \in \R$, the \emph{hyperplane} \[\hyperplane{a,b} \coloneqq \{x \in \R^d \mid \inner{a,x}+b=0\}\]subdivides $\R^d$ into \emph{half-spaces} \[\halfspace{+}{a,b} \coloneqq \{x \in \R^d \mid \inner{a,x}+b\geq 0\} \text{ and }\halfspace{-}{a,b} \coloneqq \{x \in \R^d \mid \inner{a,x}+b\leq0\}.\]
A \emph{polyhedron} $P$ is the intersection of finitely many half-spaces and a \emph{polytope} is a bounded polyhedron. A hyperplane \emph{supports} $P$ if it bounds a half-space containing $P$, and
any intersection of $P$ with such a supporting hyperplane yields a \emph{face} $F$ of $P$. For any direction $u \in \R^d$, it holds that $P^u \coloneqq \argmax \{\inner{u,y} \mid y \in P\}$ is a face of $P$ and every face of $P$ except for the empty face is of this form. 
A face is a \emph{proper face} if $F \subsetneq P$ and inclusion-maximal proper faces are referred to as \emph{facets}.
A \emph{polyhedral cone} $C \subseteq \R^d$ is a polyhedron such that $\lambda u + \mu v \in C$ for every $u,v \in C$ and  $\lambda, \mu \in \R_{\geq 0}$.
The \emph{dual cone} of $C$ is
$
	C^\vee = \{y \in (\R^d)^* \mid \inner{x,y}\geq 0 \text{ for all } x \in C\}.
$
A cone is \emph{pointed} if it does not contain a line.  A cone $C$ is \emph{simplicial}, if there are linearly independent vectors $v_1,\ldots,v_k\in \R^d$ such that $C=\{\sum_{i=1}^k\lambda_i v_i \mid \lambda_i \geq 0\}$. For a cone $C$ and $t \in \R^d$, we call $t+C$ a \emph{shifted cone} or \emph{translated cone}.
        A \emph{ray} $\rho$ is a one-dimensional pointed cone; a vector $ r$ is a \emph{ray generator} of $\rho$ if $\rho = \{\lambda\,  r \mid \lambda \geq 0\}$.
      A ray $\rho \subseteq C$ of a cone~$C$ is an \emph{extreme ray} if there do not exist $\lambda_1,\lambda_2 >0$ and distinct rays $\rho_1,\rho_2 \subseteq C$ such that $\rho = \lambda_1 \rho_1 + \lambda_2 \rho_2$.
A \emph{polyhedral complex} $\complex$ is a finite collection of polyhedra such that \begin{enumerate}
    \item $\emptyset \in \complex$
    \item if $P \in \complex$ then all faces of $P$ are in $\complex$, and 
    \item if $P, P' \in \complex$, then $P \cap P'$ is a face both of $P$ and $P'$.
\end{enumerate} 
For a polyhedral complex $\complex$ in $\R^d$ and $k \leq d$ we denote by $\complex^k$ the set of $k$-dimensional polyhedra in~$\complex$. A \emph{polyhedral fan} is a polyhedral complex in which every polyhedron is a cone. We call $\complex^{d-1}$ the \emph{facets} of $\complex$ and $\complex^d$ the \emph{regions} of $\complex$.
The \emph{dimension} of a complex~$\complex$ is the maximal dimension of its polyhedra.  
A polyhedral complex~$\complex$ in~$\R^d$ is called \emph{pure} of dimension~$k$ if all its maximal polyhedra have dimension~$k$, and \emph{complete} if the union of its polyhedra covers~$\R^d$. In this paper, we mainly consider pure and complete polyhedral complexes.

Given two $d$-dimensional polyhedral complexes $\complex, \mathcal Q$, the complex $\complex$ is a \emph{refinement} of $\mathcal Q$ if for every $\tau \in \complex^d$ there exists $\sigma \in \mathcal Q^d$ such that $\tau \subseteq \sigma$. The complex $\mathcal Q$ is a \emph{coarsening} of $\complex$ if $\complex$ is a refinement of $\mathcal Q$.
The \emph{star} of a face $\tau \in \complex$ is the set of all faces containing $\tau$, i.e., $\str_{\complex}(\tau) = \{\sigma \in \complex \mid \tau \subseteq \sigma\}$. The \emph{support} of a polyhedral complex $\complex$ is given by $|\complex|=\bigcup_{P \in \complex} P$.

\subsection{CPWL Functions}
A function $f\colon\R^d\to\R$ is called \emph{continuous and piecewise linear} (CPWL)\footnote{In our paper, we define piecewise linear functions to have finitely many pieces.}, if there exists a complete polyhedral complex $\complex$ such that the restriction of $f$ to each polyhedron $P\in\complex$ is an affine linear function. If this condition is satisfied, we say that $f$ and $\complex$ are \emph{compatible} with each other. A vector $x \in  \R^d$ is a \emph{breakpoint} of $f$ if there is no open set $U \subseteq \R^d$ containing $x$ such that $f$ is affine linear on $U$. The breakpoints of $f$ are contained in the union $\bigcup_{\sigma \in \complex^{d-1}} \sigma$ of facets of every complex $\complex$ compatible with $f$. The restriction $f_{|R}$ of the function $f$ to a region $R$ of $\complex$ is called a \emph{linear component} of $f$. The \emph{number of pieces} of $f$ is the smallest possible number of regions $\#\complex^d$ of a compatible polyhedral complex $\complex$. We denote by~$\cpwl_d$ the vector space of CPWL functions from~$\R^d$ to~$\R$.

If $f$ is a \emph{convex} CPWL function, then it can be written as the maximum of its linear components $f(x)=\max_{i\in[k]} \inner{a_i,x} + b_i$. It follows that there is a \emph{unique coarsest} compatible polyhedral complex $\complex_f$, namely the one with \[\complex_f^d = \{\{x\in \R^d\mid \inner{a_i,x} + b_i = f(x) \}\mid i \in [k]\}.\] Hence, for convex functions, the number of pieces and the number of linear components coincides. We denote by~$\cpwl^+_d$ the cone of convex CPWL functions from~$\R^d$ to~$\R$.
We call a polyhedral complex $\complex$ \emph{regular}, if there exists a convex CPWL function $f$ such that $\complex=\complex_f$.

\section{Different Perspectives on CPWL functions}
\label{sec:different_perspectives}
In this section, we describe several ways to represent and study CPWL functions from a polyhedral perspective. First, we show that CPWL functions can be parameterized by encoding their local convexity or concavity at the facets of a compatible polyhedral complex, leading to the notion of a balanced complex. Positively homogeneous convex CPWL functions, in particular, correspond to polytopes, and the associated balanced fan is the normal fan of the polytope. This section develops these correspondences in more detail. Most of the results are classical in tropical geometry (see e.g. \cite{maclagan2015introduction} and \cite{joswig21_essentialstropicalcombinatorics}) and already appear in \citep{McMullen1996}, but sometimes we adapt them slightly to fit our setting. 
\subsection{CPWL Functions and Balanced Polyhedral Complexes}

A pure $d$-dimensional polyhedral complex $\complex$ can be equipped with a \emph{weight function} $w : \complex^{d-1} \to \R$, as we describe as follows.
Given a face $\sigma \in \complex$, we denote by $\aff(\sigma) \subseteq \R^d$ the unique smallest affine subspace containing $\sigma$. The \emph{relative interior} of $\sigma$ is the interior of $\sigma$ inside the affine space $\aff(\sigma)$.
For any dimension $k\leq d$ and any $\tau \in \complex^{k-1},\sigma \in \complex^{k}$ with $\tau \subseteq \sigma$, let  $e_{\sigma/\tau} \in \R^d$ be the normal vector of $\tau$ with respect to $\sigma$, that is, the unique vector with length one that is parallel to $\aff(\sigma)$, orthogonal to $\aff(\tau)$, and points from the relative interior of $\tau$ into the relative interior of $\sigma$. 
A pair $(\complex,w)$ forms a \emph{(weighted) balanced polyhedral complex} if the weight function satisfies the \emph{balancing condition} at every  $\tau \in \complex^{d-2}$:
\begin{equation*}\label{eq:balancing-condition}
    \sum\limits_{\substack{\sigma \in \complex^{d-1}: \\ \sigma \supset \tau}} w(\sigma)\cdot e_{\sigma/\tau}=0.
\end{equation*}

In the case where $\complex$ is a polyhedral fan in $\R^2$, there is only one $0$-dimensional face, namely the origin. The balancing condition then requires that the weighted sum of the unit-length ray generators of the rays in $\complex$ equals the zero vector. In higher dimensions, taking the star of a codimension-$2$ face $\tau$ and modding out $\tau$ yields a two-dimensional fan. Intuitively, the balancing condition requires that this two-dimensional fan is balanced in the same sense as described above, see \Cref{fig:balancing_condition} for an illustration.
\begin{figure}
    \centering
    \begin{subfigure}[t]{0.31\textwidth}
\centering
    \begin{tikzpicture}[scale = 1]
    \begin{scope}[shift={(0,0)}] % Apply shift here
        \draw[convex, very thick] (2,0)-- (0,0) -- (0,2);
        \draw[concave, very thick] (0,0) -- (2,2);
        \draw[gray,dashed, very thick] (0,0) -- (-2,-2);
        %\draw[gray,dashed] (0,0) -- (0,-2);
        %\draw[gray,dashed] (0,0) -- (-2,0);
        \node at (1.5,0.75) {$x_2$};
        \node at (0.75,1.5) {$x_1$};
        \node at (0.5,-1) {$0$};
      \node at (-1,0.5) {$0$};
       %\node at (-1,-0.5) {$0$};
        %\node at (-0.5,-1) {$0$};
    \end{scope}
    \end{tikzpicture}
    \caption{The CPWL function $\max \{0, \min \{x_1,x_2\}\}$ and a compatible polyhedral complex. Blue corresponds to convex breakpoints and green to concave breakpoints.}
    \label{fig:running_example_cpwl}
\end{subfigure}
\hspace{0.1cm}
\begin{subfigure}[t]{0.31\textwidth}
  \centering
    \begin{tikzpicture}[scale = 0.7]
        \begin{scope}[scale=1,shift={(6,0)}] % Apply shift here
        \draw[convex] (0,0)-- (2,0) node[right]{$1$};
        \draw[convex] (0,0)-- (0,2) node[above]{$1$};
        \draw[concave] (0,0) -- (2,2) node[above]{$-\sqrt{2}$};
        \draw[gray,dashed] (0,0) -- (-2,-2) node[below]{$0$};
        \draw[->,black, thick] (0,0)-- (1,0);
        \draw[->,black, thick] (0,0)-- (0,1);
        \draw[->,black, thick] (0,0) -- (0.7,0.7);
        \draw[->,black, thick] (0,0) -- (-0.7,-0.7);
    \end{scope}
    \end{tikzpicture}
    \caption{A balanced fan where the weight function is induced by the CPWL function given in \Cref{fig:running_example_cpwl} as described in \Cref{lem:structure_theorem_tropical_geometry}.}
    \end{subfigure}
    \hspace{0.1cm}
\begin{subfigure}[t]{0.31\textwidth}
  \centering
    \begin{tikzpicture}[scale=2.75]
    % Draw the central 2D ridge face

    % Draw the 3 incident polyhedral cones
    \fill[convex, opacity=0.1] (0, 0, 0) -- (1, 0, 0) -- (0, 1, 0) -- cycle;
     \fill[convex, opacity=0.1] (0, 0, 0) -- (-1 ,0, -1) -- (0, 1, 0) -- cycle;
    \fill[convex, opacity=0.1] (0, 0, 0) -- (0, 0, 1) -- (0, 1, 0) -- cycle;
    \draw[black,thick] (0, 0, 0) -- (0, 1, 0) node[above] {$\tau$};

      \draw[->] (0,0.25,0) -- (0,0.25,0.5) node[below]{$e_{\sigma_3 / \tau} $};
   \draw[->] (0,0.25,0) -- (0.25,0.25,0)node[right]{$e_{\sigma_1 / \tau} $};
   \draw[->] (0,0.25,0) -- (-0.25,0.25,-0.25)node[left]{$e_{\sigma_2 / \tau} $};

    \node at (1, 0, 0) {$\sigma_1$};
    \node at (-1, 0, -1) {$\sigma_2$};
    \node at (0, 0, 1) {$\sigma_3$};

\end{tikzpicture}
\caption{Illustration of the balancing condition in higher dimension. The weighted sum of the vectors $e_{\sigma_i / \tau}$ should equal the zero vector.}
\end{subfigure}
\caption{An illustration of the balancing condition.}
\label{fig:balancing_condition}
\end{figure}

\begin{lemma}
\label{lem:structure_theorem_tropical_geometry}
   Let $f$ be a CPWL function compatible with the polyhedral complex $\complex$. For a facet $\sigma \in \complex^{d-1}$, let $P,Q \in \complex^d$ be the unique polyhedra such that $P \cap Q = \sigma$ and let $a_P,a_Q \in \R^d, b_P,b_Q\in \R$ such that $f_{|P}(x) = \inner{a_P,x} + b_P$ for $x \in P$ and $f_{|Q}(x) = \inner{a_Q,x} + b_Q$  for $x \in Q$. Then, the function $w_f \colon \complex^{d-1} \to \R$ with \[w_f(\sigma)  \coloneqq \langle e_{P/\sigma}, a_{P} \rangle + \langle e_{Q/\sigma}, a_{Q} \rangle = \langle e_{P/\sigma}, a_{P} - a_{Q} \rangle \]
 defines a balanced polyhedral complex $(\complex,w_f)$.
\end{lemma}
\begin{proof}
     Note that if $f$ is locally convex at $\sigma$, then $\langle e_{P/\sigma}, a_{P} - a_{Q} \rangle = \|a_{P} - a_{Q}\|_2$ and if $f$ is locally concave at $\sigma$, then $\langle e_{P/\sigma}, a_{P} - a_{Q} \rangle = -\|a_{P} - a_{Q}\|_2$. The proof proceeds analogously to the case where $f$ has only convex breakpoints and the coefficients of the affine maps are rational. In this case, the lemma follows from the structure theorem of tropical geometry (\citet{maclagan2015introduction} Proposition 3.3.2).
    Here, we present an adjusted proof (to not necessarily convex functions and irrational coefficients).
 Let $\tau \in \complex^{d-2}$ and $\{P_1,\ldots,P_m\}=\str_\complex(\tau) \cap \complex^d$ and $\{\sigma_1,\ldots,\sigma_m\}=\str_\complex(\tau) \cap \complex^{d-1}$ be ordered in a cyclic way, that is, $P_i \cap P_{i+1} = \sigma_i$ for $i \in [m]$, where $P_{m+1}=P_1$.  Note that, since $f$ is continuous, we have that $a_{P_i}-a_{P_{i+1}} \in \spn(e_{P_i/\sigma_i})$. The linear map $T_\tau \colon \aff(\tau)^\perp \to \aff(\tau)^\perp$ satisfying $T_\tau(e_{P_i/\sigma_i}) = e_{\sigma_i/\tau}$ (given by a rotation matrix) is an automorphism, implying that
    \begin{align*}
    \sum_{\substack{\sigma \supset \tau\\\sigma \in \complex^{d-1}}} w_f(\sigma)\cdot e_{\sigma/\tau}&=
    \sum_{i=1}^{m}w_f(\sigma_i)\cdot T_\tau(e_{P_i/\sigma_i})\\&=
    T_\tau\left(\sum_{i=1}^{m}\inner{e_{P_i/\sigma_i},a_{P_i}-{a_{P_{i+1}}}}\cdot e_{P_i/\sigma_i}\right)\\&=
    T_\tau\left(\sum_{i=1}^{m}\inner{e_{P_i/\sigma_i},e_{P_i/\sigma_i}}\cdot (a_{P_i}-{a_{P_{i+1}}})\right)\\&=
    T_\tau(0)=0.
    \end{align*}
\end{proof}

\begin{lemma}
\label{lem:existence_of_function_for_balanced_complex}
    Let $(\complex,w)$ be a balanced polyhedral complex. Then there exists a function $f$ compatible with $\complex$ such that $w_f=w$.
\end{lemma}

\begin{proof}
In the case that $w_f$ is nonnegative, the lemma follows from \citet{maclagan2015introduction} Proposition 3.3.10. Here, we present an adjusted proof (to not necessarily convex functions and irrational coefficients).
Let $\complex^d=\{P_1, \ldots, P_k\}$ and for $P \in \complex^d, \sigma \in \complex^{d-1}$, let $b_{P/\sigma} \in \R$ such that $\sigma$ is contained in the hyperplane $\{x \in \R^d \mid \inner{e_{P/\sigma},x}+b_{P/\sigma}=0\}$ and define the function $f_{P/\sigma}\colon\R^d \to \R$ by $f_{P/\sigma}(x)=\inner{e_{P/\sigma},x}+b_{P/\sigma}$. 
Since $\complex$ is complete, the graph $G=(V,E)$ given by $V=\{1,\ldots,k\}$ and $E=\{\{i,j\}\mid P_i \cap P_j \in \complex^{d-1}\}$ is connected. Start by defining the function $f_{|P_1} = 0$. For $1<i\leq k$, let $(j_1,\ldots, j_m)$ be a path from $1$ to $i$ and for $\ell \in[m-1]$, let $\sigma_\ell = P_{j_\ell} \cap P_{j_{\ell+1}}$ and define the function \[f_{|P_i}= \sum_{\ell=1}^{m-1} w(\sigma_\ell) \cdot f_{P_{j_\ell}/\sigma_\ell}.\] 
    First, we argue that $f_{|P_i}$ is well-defined, that is, the definition of $f_{|P_i}$ does not depend on the path from vertex $1$ to $i$. Equivalently, it suffices to show that for any cycle $(i_1,\ldots,i_m)$ in $G$ with $i_1=i_m$, it holds that
        $\sum_{\ell=1}^{m-1} w(\sigma_\ell) f_{{P_{i_\ell}/\sigma_\ell}}=0$, where $\sigma_\ell = P_{i_\ell} \cap P_{i_{\ell+1}}$. Since $\complex$ is complete any cycle decomposes into cycles $(i_1,\ldots,i_m)$ corresponding to the star of a cone $\tau \in \complex^{d-2}$, that is, $\{\sigma_1,\ldots, \sigma_{m-1}\} = \{\sigma \in \complex^{d-1} \mid \sigma \supset \tau\}$.
        Since the map $T_\tau$ from \Cref{lem:structure_theorem_tropical_geometry} is an automorphism, it holds that $\sum_{\ell=1}^{m-1} w(\sigma_\ell)\cdot e_{\sigma_\ell/\tau}=0$ if and only if $\sum_{\ell=1}^{m-1} w(\sigma_\ell) \cdot e_{P_{i_\ell}/\sigma_\ell}=0$.

        So, let $x \in \R^d$ be arbitrary and $x' \in \aff(\tau)$ and $x''\in \aff(\tau)^\perp$ such that $x=x'+x''$. Since $w$ is balanced, it holds that $\sum\limits_{\substack{\sigma \supset \tau\\\sigma \in \complex^{d-1}}} w(\sigma)\cdot e_{\sigma/\tau}=0$ and hence it follows that \begin{align*}
           \sum_{\ell=1}^{m-1} w(\sigma_\ell) f_{{P_{i_\ell}/\sigma_\ell}}(x) &= \sum_{\ell=1}^{m-1} w(\sigma_\ell) \cdot (\inner{e_{P_{i_\ell}/\sigma_\ell} ,x'+x''}+b_{P_{i_\ell}/\sigma_\ell})
           \\&= \sum_{\ell=1}^{m-1} w(\sigma_\ell) \cdot \inner{e_{P_{i_\ell}/\sigma_\ell} ,x''}
           \\&=  \inner{\sum_{\ell=1}^{m-1} w(\sigma_\ell) \cdot e_{P_{i_\ell}/\sigma_\ell,x''}}
           \\&=0.
        \end{align*}

        By definition, $f$ is a CPWL function  and compatible with $\complex$. Let $P,Q \in \complex^{d}$ such that $\sigma = P \cap Q \in \complex^{d-1}$. Then it holds that $a_P-a_Q=w(\sigma)\cdot e_{P/\sigma}$ and hence \[w_f(\sigma)=\inner{e_{P/\sigma}, a_{P}} + \inner{e_{Q/\sigma}, a_{Q}}= \inner{e_{P/\sigma},a_P-a_Q} =\inner{e_{P/\sigma},w(\sigma)\cdot e_{P/\sigma}} = w(\sigma),\]
        finishing the proof.
\end{proof}

\begin{lemma}
\label{lem:VP_subspace}
    Let $\complex$ be a polyhedral complex in $\R^d$. The set of CPWL functions compatible with $\complex$ forms a linear subspace $\VC$ of the vector space $\cpwl_d$ of CPWL functions from~$\R^d$ to~$\R$. 
\end{lemma}

\begin{proof}
    Let $f,g$ be CPWL functions which are compatible with $\complex$, and $\lambda, \mu \in \R$. Then for any $P \in \complex^d$ holds
    $(\lambda f + \mu g)_{|P} = \lambda f_{|P} + \mu g_{|P}$, which is an affine function restricted to $P$. Thus, the set $\VC$ of CPWL functions compatible with $\complex$ forms a linear subspace of the space of continuous functions.
\end{proof}

Let $\Aff(\R^d)$ be the space of affine functions from $\R^d$ to $\R$. For many of our arguments, adding or subtracting an affine function $a\in \Aff(\R^d)$ does not change anything. In particular, a function $f$ is convex if and only if $f+a$ is convex. Therefore it makes sense to define the quotient space $\VCmod\coloneqq\VC/\Aff(\R^d)$, where we identify functions in $\VC$ that only differ by adding an affine function.
\Cref{lem:structure_theorem_tropical_geometry} and \Cref{lem:existence_of_function_for_balanced_complex} imply that we can parameterize a function $f \in \VCmod$ by keeping track of ``how convex or concave'' the function is at the common face $\sigma \in \complex^{d-1}$ of two neighboring pieces, as we summarize in the following lemma.

\begin{lemma}\label{lem:edge_rep}
The vector space 
$\WC \coloneqq \{w \colon \complex^{d-1} \to \R \mid (\complex,w)$ is balanced$\}$
is isomorphic to $\VCmod$.
\end{lemma}

\begin{proof}
     By \Cref{lem:structure_theorem_tropical_geometry}, $f \mapsto w_f$ is a linear map from $\VC$ to $\WC$, which is surjective by \Cref{lem:existence_of_function_for_balanced_complex}. The kernel of the map is $\Aff(\R^d)$. To see this, let $f$ be in the kernel of the map $f \mapsto w_f$. It holds that \[w_f(\sigma) = \inner{e_{P/\sigma},a_P}+\inner{e_{Q/\sigma},a_Q} = \inner{e_{P/\sigma},a_P-a_Q} = 0\] if and only if $a_P-a_Q=0$ due to the fact that $a_P-a_Q \in \spn(e_{P/\sigma})$. Due to the continuity of $f$, this also implies that $b_{P}=b_Q$ and hence $f_{|Q}=f_{|P}$ and therefore $f$ is affine linear.
\end{proof}
\begin{corollary}
\label{cor:VP_finite_dim}
    $\VC$ is finite-dimensional.
\end{corollary}
\begin{proof}
    By \Cref{lem:edge_rep}, we have that 
    $\VCmod = \VC / \Aff(\R^d) \cong \WC$. Thus, the dimension of $\VC$
    is bounded from above by $\dim(\WC)+\dim(\Aff(\R^d))\leq \#\complex^{d-1}+(d+1)$.
\end{proof}
In the following proposition, we prove that nonnegative (positive) weights in fact correspond to the function being (strictly) convex.
\begin{proposition}
\label{lem:nonnegative}
    A function $f \in \VC$ is convex if and only if $w_f$ is nonnegative. Moreover, $f$ is strictly convex with $\complex = \complex_f$ if and only if $w_f$ is strictly positive.
\end{proposition}
\begin{proof}
The function $f$ is convex if and only if it is locally convex around every $x \in \R^d$. If $x$ is in the relative interior of some $P \in \complex^d$, then this is clearly satisfied since the function is locally affine linear. Now, assume that $f$ is not locally convex around a $x \in \tau$ for some $\tau \in \complex^{d-2}$. In other words, there are $y,z \in \R^d$ and $\lambda \in (0,1)$ such that $f(\lambda z + (1-\lambda)y) > \lambda f(z) + (1-\lambda)f(y)$ and such that the line between $x$ and $y$ intersects $\tau$. Let $L$ be the Lipschitz constant of $f$ and $\delta \coloneqq f(\lambda z + (1-\lambda)y) - \lambda f(z) - (1-\lambda)f(y) > 0$. Let $\varepsilon \coloneqq \frac{\delta}{4L} > 0$. Then there are $v,w \in \R^d$ with $\|v\|,\|w\| < \varepsilon$ such that the line between $z+v$ and $y+w$ does not intersect any face $\tau \in \complex^{d-2}$. But then,

\begin{align*}
    f(\lambda (z+v) + (1-\lambda)(y+w))&\geq f(\lambda z + (1-\lambda)y) - L (\|\lambda v \| + \|(1-\lambda)w\|)
    \\& > f(\lambda z + (1-\lambda)y) - 2L\varepsilon
    \\&= \delta +\lambda f(z) + (1-\lambda)f(y) - 2L\varepsilon
    \\&\geq \delta +  \lambda f(z+v) + (1-\lambda)f(y+w) - 2L\varepsilon - 2L\varepsilon
    \\&=  \lambda f(z+v) + (1-\lambda)f(y+w)
\end{align*}

and there must be a $x'$ in the relative interior of some $\sigma \in \complex^{d-1}$ with $\sigma =P \cap Q$ for $Q \in \complex^d$ such that $f$ is not locally convex around $x'$.
Hence, $f$ is convex if and only if $f$ is locally convex around every
$\sigma \in \complex^{d-1}$, that is, $f$ is locally convex around every $x$ in the relative interior of $\sigma$. For any such $x$, there is a $\lambda >0$ such that  $x + \lambda \cdot e_{P/ \sigma} \in P$ and $x + \lambda \cdot e_{Q/ \sigma} \in Q$, by construction of $e_{P/ \sigma}$ and $e_{Q/ \sigma}$. Recall from the proof of \Cref{lem:edge_rep} that $w_f(\sigma)=\inner{e_{P/\sigma}, a_{P}} + \inner{e_{Q/\sigma}, a_{Q}} $, where $f_{|P}(x) = \langle a_P,x\rangle +b_P$ and $f_{|Q}(x) = \langle a_Q,x\rangle +b_Q$. Since $P,Q \in \complex^d$ and $\| e_{P/\sigma} \| = \|e_{Q/\sigma}\| = 1$, we have that $x$ is the midpoint of $x + \lambda \cdot e_{P/\sigma}$ and $x + \lambda \cdot e_{Q/\sigma}$. Therefore, 
$f$ is convex if and only if $f(x) \leq \frac{1}{2} f(x + \lambda \cdot e_{P/\sigma}) +
\frac{1}{2} f(x + \lambda \cdot e_{Q/ \sigma})$. Equivalently,
        \[0 \leq f(x + \lambda \cdot e_{P/\sigma}) + f(x + \lambda \cdot e_{Q/ \sigma}) - 2f(x)   = \lambda(\langle e_{P/\sigma}, a_{P} \rangle + \langle e_{Q/\sigma}, a_{Q} \rangle)=\lambda \cdot  w_f(\sigma).\]
    If $\complex = \complex_f$, then we have strict local convexity at every $\sigma \in \complex^{d-1}$, which means a strict inequality in the inequality above. 
\end{proof}
\begin{lemma}
\label{lem:VP_cone}
  The set $\VC^+$ of \emph{convex} functions in $\VC$ forms a polyhedral cone and \[\VCP \coloneqq \VC^+/\Aff(\R^d) \] forms a polyhedral cone in $\VCmod$.
\end{lemma}
\begin{proof}
    \Cref{lem:edge_rep} and \Cref{lem:nonnegative} imply that the set of convex functions in $\VCmod$ satisfies 
    \[\VCP \cong \WCP \coloneqq \bigcap_{\sigma \in \complex^{d-1}}\{w \in \WC \mid w(\sigma) \geq 0\}.\] This is a finite intersection of linear inequalities, so $\WC^+$ is a polyhedral cone. Moreover, ``$\cong$'' is a linear isomorphism, which implies that $\VCP$ is a polyhedral cone. The first statement thus follows from $\VC^+ = \VCP+\Aff(\R^d)$.
\end{proof}
\begin{example}[Median]\label{ex:parameterizations}
   In this example, we have a look at the different parameterizations for the function that computes the median of $3$ numbers. To have a $2$-dimensional example, we set $x_3 \coloneqq 0$. The cones of the form \[P_{\pi}=\{x \in \R^2 \mid x_{\pi(1)} \leq x_{\pi(2)} \leq x_{\pi(3)}\}\] are the maximal cones of the braid fan (\Cref{def:braid_arrangement}), where $\pi \colon [3] \to [3]$ is a permutation. Then, let $f \colon \R^2 \to \R$ be the function given by $f_{|P_{\pi}}(x) = x_{\pi(2)}$. The function $f$ has concave breakpoints whenever the median and the higher coordinate change, i.e., at the facets that are given as $\sigma_{\pi,1}=\{x \in \R^2 \mid x_{\pi(1)} \leq x_{\pi(2)} = x_{\pi(3)}\}$
   and convex breakpoints whenever the median and the lower coordinate change, that is, at the facets that are given as \mbox{$\sigma_{\pi,2}=\{x \in \R^2 \mid x_{\pi(1)} = x_{\pi(2)} \leq x_{\pi(3)}\}$.}  Since $\|e_1-e_2\|_2 =\sqrt{2}$ and $\|e_1\|_2=\|e_2\|_2=1$,
   it holds that \[w_f(\sigma_{\pi,1})= 
   \begin{cases}
       -\sqrt{2} & x_{\pi(1)}=x_3 \\ 
       -1 & x_{\pi(1)} \neq x_3
   \end{cases} \quad \text{ and } \quad
   w_f(\sigma_{\pi,2})= 
   \begin{cases}
       \sqrt{2} & x_{\pi(3)}=x_3 \\ 
       1 & x_{\pi(3)} \neq x_3.
   \end{cases}\]
   See \Cref{fig:3dim_median} for a $2$-dimensional illustration.
   \end{example}
\begin{figure}
\centering
\begin{subfigure}[t]{.46\textwidth}
\centering
    \begin{tikzpicture}[scale=1.2]
        % \draw[concave] (0,0) -- (2,0) node[right]{$\sqrt{2}$};
        % \draw[convex]  (0,0) --  (-2,0) node[left]{$-\sqrt{2}$};
        % \draw[convex] (0,0) -- (1,-1.732) node[below]{$\sqrt{2}$};
        % \draw[concave]  (0,0) -- (-1,1.732) node[above]{$-\sqrt{2}$};
        % \draw[convex] (0,0) -- (1,1.732) node[above]{$\sqrt{2}$};
        % \draw[concave] (0,0) -- (-1,-1.732) node[below]{$-\sqrt{2}$};
        \draw[convex] (0,0) -- (0,2) node[above]{$1$};
        \draw[concave, dashed] (0,0) -- (0,-2) node[below]{$-1$};
        \draw[convex] (0,0) -- (2,0) node[right]{$1$};
        \draw[concave, dashed] (0,0) -- (-2,0) node[left]{$-1$};
        \draw[convex] (0,0) -- (-2,-2) node[below]{$\sqrt{2}$};
        \draw[concave, dashed] (0,0) -- (2,2) node[above]{$-\sqrt{2}$};

        \node at (0.9,1.8) {\small $0\leq x_1 \leq x_2$};
        \node at (1.9,0.7) {\small $0\leq x_2 \leq x_1$};
        \node at (-1,1) {\small $x_1 \leq 0 \leq x_2$};
        \node at (1,-1) {\small $x_2 \leq 0 \leq x_1$};
        \node at (-0.9,-1.8) {\small $ x_2 \leq x_1 \leq 0$};
        \node at (-1.9,-0.7) {\small $x_1 \leq x_2 \leq 0$};
    \end{tikzpicture}
      \caption{Parameterization of the median function via the weights on the $1$-dimensional facets.}
      \label{fig:parameterization-weights}
    \end{subfigure}
    \hspace{1em}
\begin{subfigure}[t]{.46\textwidth}
\centering
    \begin{tikzpicture}[scale=1.4]
    \definecolor{convex}{rgb}{0,0.3,0.5}
    \definecolor{concave}{rgb}{0.8,0.2,0}
        \draw[black] (0,0) -- (0,2) ;
        \draw[black] (0,-2) -- (0,0);
        \draw[black] (0,0) -- (2,0);
        \draw[black] (-2,0) -- (0,0);
        \draw[black] (-2,-2) -- (0,0);
        \draw[black] (0,0) -- (2,2);

        \node at (0.75,1.5) {$x_1$};
        \node at (1.5,0.75) {$x_2$};
        \node at (-1,1) {$0$};
        \node at (1,-1) {$0$};
        \node at (-0.75,-1.5) {$x_1$};
        \node at (-1.5,-0.75) {$x_2$};
    \end{tikzpicture}
    \caption{Parameterization of the median function via its linear maps on the maximal polyhedra.}
    \end{subfigure}
     \caption{Two different parameterizations of the function  that computes the median of $\{0,x_1,x_2\}$. In \Cref{fig:parameterization-weights}, the convex breakpoints are colored in blue, and concave breakpoints are dashed and colored in red.}
     \label{fig:3dim_median}
\end{figure}

\subsection{Duality and Newton Polytopes}
In this section, we describe the duality between positively homogeneous convex CPWL functions and polytopes. The \emph{Minkowski sum} of two polytopes $P,Q \subseteq \R^d$ is defined as $P+Q\coloneqq\{p+q:p\in P,q\in Q\}$.

A CPWL function $f \colon \R^d \to \R$ is \emph{positively homogeneous} if $f(\lambda x) = \lambda f(x)$ for all $x \in \R^d$ and $\lambda \geq 0$.  If $f$ is convex, then it is positively homogeneous if and only if $f(0)=0$, and $\complex_f$ is a polyhedral fan. In this case, it can be written as $f(x) = \max_{i \in [k]} \inner{x,a_i}$, where $a_i \in \R^d$. 
We define the \emph{Newton polytope} of such a function $f$ as the convex hull $\Newt(f) = \conv(a_1,\dots,a_k)$. Conversely, the support function $f_P$ of a polytope $P$ is given by $f_P(x) = \max_{y \in P} \inner{y,x}$. See \Cref{fig:newton_polytope} for an example.

\begin{theorem}[see, e.g., \citet{schneider2013convex,hertrich2021towards}]
\label{th:iso_polytopes_ccpwl}
    For all  convex positively homogeneous CPWL functions $f\colon \R^d \to \R$, it holds that $f$ is the \emph{support function} of $\Newt(f)$.
    Moreover, it holds that $\Newt(f) + \Newt(g) =\Newt(f+g)$ and $\Newt(\lambda f) = \lambda \Newt(f)$ for all convex positively homogeneous CPWL functions $f,g \colon \R^d \to \R$ and all $\lambda \geq 0$.
\end{theorem}

We now give an interpretation of $\complex_f$ and $w_f$ 
in terms of the Newton polytope.
Given any $d$-dimensional polytope $P \subset \R^{d}$, the (outer) normal cone of a $k$-dimensional face $F$ of $P$ is the $(d-k)$-dimensional cone
\begin{equation}\label{eq:normal-cone}
    N_F(P) = \left\{x \in \R^{d} \ \middle| \ \inner{z,x} = \max_{y \in P} \inner{y,x} \text{ for all } z \in F \right\}. 
\end{equation}
In particular, if $P = \Newt(f)$ and $v$ is a vertex of $P$, then the description of the normal cone turns into
\[
    N_v(\Newt(f)) = \left\{x \in \R^{d} \ \middle| \ \inner{v,x} = f(x) \right\},
\]
and agrees with a polyhedron in $\complex_f^d$.
The \emph{normal fan} of a polytope is the collection of normal cones over all faces. Thus, for positively homogeneous convex functions, the polyhedral complex $\complex_f$ agrees with the normal fan of $\Newt(f)$, and the number of linear pieces of $f$ equals the number of vertices of $\Newt(f)$. An example is given in \Cref{fig:newton_polytope}.
The duality between $\complex_f$ and $\Newt(f)$ also establishes a bijection between facets $\sigma \in \complex_f^{d-1}$ and edges of $\Newt(f)$, and for the corresponding weight function $w_f$ holds that $w_f(\sigma)$ equals the Euclidean length of the edge that is dual to $\sigma$.
This correspondence extends to general convex CPWL functions, where $\complex_f$ is a complex which is dual to a regular polyhedral subdivision of $\Newt(f)$,
and $w_f$ corresponds to lengths of edges in this subdivision (see e.g. \cite[Chapter 3.4]{maclagan2015introduction}).
  \begin{figure}
      \centering
      \begin{tikzpicture}[scale=2]
      \begin{scope}[scale = 0.5,shift={(0,0)}]
     \draw[convex] (2,0)-- (0,0) -- (0,2);
        \draw[convex] (0,0) -- (-2,-2);
        \node at (1.5,1) {$x_1+x_2$};
          \node at (0.5,-1) {$x_1$};
      \node at (-1,0.5) {$x_2$};
  \draw[convex] (0,0)-- (0,2) node[above] {$1$};
        \draw[convex] (0,0) -- (-2,-2) node[below] {$\sqrt{2}$};
        \draw[convex] (0,0)-- (2,0) node[right] {$1$};

        \end{scope}

          \begin{scope}[shift={(3,-0.5)}] 
        \draw[convex] (1,1) -- (1,0) node[below]{\color{black}{$e_1$}};
        \draw[convex] (1,0) -- (0,1) node[left]{\color{black}{$e_2$}};
        \draw[convex] (0,1) -- (1,1) node[right]{\color{black}{$e_1+e_2$}};
        \end{scope}
          \begin{scope}[shift={(3,-0.5)}] 
        \node[convex] at (0.2,0.2) {\small{$\sqrt{2}$}};
        \draw[convex] (1,1) -- (1,0) node[midway, right] {\small{$1$}};
    \draw[convex] (1,0) -- (0,1);
    \draw[convex] (0,1) -- (1,1) node[midway, above] {\small{$1$}};
      
        \end{scope}
      \end{tikzpicture}
      \caption{On the left, the positively homogeneous CPWL function $\max \{x_1,x_2,x_1+x_2$\} and its unique coarsest compatible complex. On the right, the Newton polytope $\conv(e_1,e_2,e_1+e_2)$ of the function. The fan is the normal fan of this polytope and the edge lengths correspond to the weights on the facets induced by the CPWL function.}
      \label{fig:newton_polytope}
  \end{figure}
\subsection{Deformation Cones}
We call a polytope $Q$ a \emph{deformation} of a polytope $P$ if there exists a polytope $R$ and a scalar $\lambda > 0$ such that $\lambda P = Q + R$. Equivalently, $Q$ is a deformation of $P$ if and only if the normal fan of $Q$ is a coarsening of the normal fan of $P$.

The \emph{deformation cone} $\defcone(P)$ of a polytope $P$ consists of all its deformations, which can be parameterized by translation and edge lengths. Specifically, assign a nonnegative real value to each edge of $P$ such that for all $2$-faces form actual polygons, that is, the balancing condition is met. The resulting polytope is a deformation of $P$, and all deformations of $P$ arise in this way. The linear condition that the edge lengths form polygons for the $2$-faces is equivalent to the condition that the corresponding weight function is balanced \citep{McMullen1996}.  In other words, if we mod out translation, the deformation cone modulo translation $\defconemod(P)$ becomes a pointed polyhedral cone, isomorphic to $\WP{\complex}$ where $\complex$ is the normal fan of $P$. The following lemma summarizes the discussion.
\begin{lemma}
\label{lem:iso_defcone_WC+_VC+}
    Let $P$ be a polytope and $\complex$ the normal fan of $P$. Then, the following three cones are linearly isomorphic:
    \begin{enumerate}
        \item the deformation cone modulo translation $\defconemod(P)$
        \item the cone $\WP{\complex}$ of all balanced nonnegative weights on $\complex$ and
        \item the cone $\VPmod{\complex}$ of convex CPWL functions compatible with $\complex$ modulo affine linear functions.
    \end{enumerate}
\end{lemma}
 Although general positively homogeneous CPWL functions may not admit a Newton polytope due to potential concave folds (which would correspond to negative edge lengths), the isomorphism between Newton polytopes and positively homogeneous convex CPWL functions extends to a linear isomorphism through the notion of virtual polytopes, which are formal Minkowski differences of polytopes.
\subsection{Virtual polytopes}
We now review the notion of virtual polytopes adapted to our setting. For a thorough introduction to virtual polytopes, we refer to \citep{panina2015virtual}. The set of polytopes $\polytopes_d$ in $\R^d$ together with Minkowski addition forms a commutative semigroup, having the $0$-dimensional polytope that consists only of the zero vector as neutral element. For any polytopes $P,Q,R \in \polytopes_d$, there holds the cancellation law, which says that if $P+Q=R+Q$, then $P=R$. This allows to construct the Grothendieck group $\vpolytopes_d \coloneqq\{(P, Q) : P, Q \in \polytopes_d\} / \sim$ where the equivalence relation $\sim$ is defined by
\begin{equation*}
    (P_1, Q_1) \sim (P_2, Q_2) \iff  P_1 + Q_2 = P_2 + Q_1.
\end{equation*}
We write the equivalence class represented by $(P,Q)$ as $P-Q$ and call it a \emph{virtual polytope}. Every CPWL function $f$ is the difference of two convex functions $g-h$ and it is easy to check that the virtual polytope $\Newt(g) - \Newt(h)$ is independent of the choice of convex functions $g$ and $h$ allowing to define $\Newt(f) \coloneqq \Newt(g) - \Newt(h)$.
 The set of virtual polytopes is a vector space with the scalar multiplication given by $\lambda  (P-Q) \coloneqq \lambda P - \lambda Q$ which is isomorphic to the vector space $\cpwl^0_d$ of positively homogeneous CPWL functions.
 
\begin{proposition}
    The map $\Newt \colon \cpwl^0_d \to \vpolytopes_d$, $f \mapsto \Newt(f)$ is a vector space isomorphism.
\end{proposition}
\begin{proof}
    Follows immediately from \Cref{th:iso_polytopes_ccpwl}.
\end{proof}

\section{Decomposition Polyhedra}\label{sec:compatible}

In this section, we introduce and generally study the main concept of this paper, decomposition polyhedra. These polyhedra describe the set of possible decompositions of a CPWL function $f$ into a difference $f=g-h$ that are compatible with a given polyhedral complex.

\begin{definition}
    For a CPWL function $f$ and a polyhedral complex $\complex$, the \emph{decomposition polyhedron} of~$f$ with respect to~$\complex$ is
    $
        \decomp{\complex}{f}\coloneqq\{(g,h)\mid g,h\in\VCP, f=g-h\}.
    $
\end{definition}

The projection $\pi((g,h))=g$ induces an isomorphism between $\decomp{\complex}{f}$ and $\pi(\decomp{\complex}{f})$ since
$
    \decomp{\complex}{f} = \{(g,g-f) \mid g \in \pi(\decomp{\complex}{f}) \}.
$
We now show that $\decomp{\complex}{f}$ is indeed a polyhedron, which arises as the intersection of two shifted copies of a cone.
\begin{proposition}\label{prop:shifted-cones}
    The set $\decomp{\complex}{f}$ is a polyhedron that arises as the intersection of convex functions with the affine hyperplane $H_f = \{(g,h) \mid f = g-h\}$, namely
    \[
        \decomp{\complex}{f} = (\VCP \times \VCP) \cap H_f.
    \]
    Under the bijection $\pi$, the decomposition polyhedron is the intersection of two shifted copies of the polyhedral cone $\VCP$. More specifically,
     $
        \pi(\decomp{\complex}{f}) = \VCP \cap (\VCP + f).
    $
\end{proposition}
\begin{proof}
    For the set of decompositions holds
    \begin{align*}
        \decomp{\complex}{f} 
        = \{(g,h) \mid g \in \VCP, h \in \VCP, f = g - h\} 
        = (\VCP \times \VCP) \cap H_f.
    \end{align*}
    For the projection, we have
    \begin{align*}
        \pi(\decomp{\complex}{f})
        &= \pi(\{(g,g-f) \mid g \in \VCP, g-f \in \VCP\})\\
        &= \{g \mid g \in \VCP, g \in f+\VCP\} 
        = \VCP \cap (f+ \VCP).\qedhere
    \end{align*}
\end{proof}

\begin{figure}[b]
    \centering
 \begin{tikzpicture}[scale=0.8]

    % Original Cone
    \draw[black] (0,3) -- (0,0) -- (3,0);
    \fill[mpblue,opacity=0.1] (0,3) -- (0,0) -- (3,0) -- cycle;
    \node[mpblue] at (1,4) {$\VCP$};
     
     % Shift Vector
    \draw[->, thick, dashed] (0,0) -- (1,-1.5) node[midway,left]{$f$};

    % Shifted Cone
    \begin{scope}[shift={(1,-1.5)}]
        \draw[black] (0,3) -- (0,0) -- (3,0);
    \fill[mpred,opacity=0.1] (0,3) -- (0,0) -- (3,0)-- cycle;
\node[mpred] at (4,1) {$f+\VCP$};
    \end{scope}

     \begin{scope}[shift={(1,0)}]
        \draw[black,very thick] (0,3) -- (0,0) -- (3,0);
    \fill[mpgreen] (0,3) -- (0,0) -- (3,0)-- cycle;
    \node[mpgreen] at (2.3,2.3) {$\pi(\decomp{\complex}{f})$};
    \fill[black] (0,0);
    \end{scope}

\end{tikzpicture}
\caption{An illustration of the decomposition polyhedra arising as the intersection of two shifted cones.}
\end{figure}

\begin{remark}
    Under the isomorphism of Lemma~\ref{lem:edge_rep}, we identify $\decomp{\complex}{f}$ with the polyhedron $\{(w_g,w_h) \in \WCP\times \WCP \mid w_g-w_h=w_f\}$ in $\WC \times \WC$ and $\pi(\decomp{\complex}{f})$ with the polyhedron $\{w_g \in \WCP \mid w_g \geq w_f\}$, where $\WCP = \{w \in \WC \mid w \geq 0\}$.
\end{remark}

For the remainder of this section, we analyze the faces of the polyhedron $\decomp{\complex}{f}$ in terms of the properties of the corresponding decompositions.

\begin{definition}
    A decomposition $(g,h)\in\decomp{\complex}{f}$ is called \emph{reduced}, if there is no convex function $\phi\in\VCP\setminus\{0\}$ such that $g-\phi$ and $h-\phi$ are both convex.
\end{definition}

If a decomposition is not reduced, then we can obtain a ``better'' decomposition by simultaneously simplifying both $g$ and $h$ through subtracting a convex function $\phi$. Hence, it makes sense to put a special emphasis on reduced decompositions. Conveniently, the following theorem links this notion to the geometry of $\decomp{\complex}{f}$.

\begin{theorem}\label{th:decomp-reduced}
    A decomposition $(g,h)\in\decomp{\complex}{f}$ is reduced if and only if $(g,h)$ is contained in a bounded face of $\decomp{\complex}{f}$.
\end{theorem}

The statement follows from a more general statement about polyhedra. Recall that any polyhedron $P$ can written as the Minkowski sum 
\[
    P = Q + \recc(P)  = \{ q + c \mid q \in Q, c \in \recc(P) \}
\]
where $Q$ is a bounded polytope, and $\recc(P)=\{v \in \R^d \mid x+tv \in P $ for all $ x\in P, t\geq0\}$ the \emph{recession cone} of $P$, see~\cite{ziegler_lecturespolytopes}. 

\begin{proposition}\label{th:bounded-faces}
    A point $x \in P$ is contained in a bounded face of $P$ if and only if $x - c \not\in P \ \forall c \in \recc(P) \setminus \{0\}$.
\end{proposition}
\begin{proof}
    Any nonempty face of the polyhedron $P$ is of the form
    \[
        P^u = \{ x \in P \mid \langle x, u \rangle \geq \langle y, u \rangle \ \forall y \in P \},
    \]
    and for Minkowski sums holds $P^u = Q^u + \recc(P)^u$.
    Let $x \in P$ be a point contained in a bounded face $P^u$ of $P$. Since $P^u$ is bounded, we have that $P^u = Q^u + \recc(P)^u$ with $\recc(P)^u = \{0\}$ being the unique bounded face of $\recc(P)$. Thus, $\langle c, u \rangle < \langle 0, u \rangle$ for all $c \in \recc(P) \setminus \{0\}$. This implies that
    $\langle x - c, u \rangle = \langle x, u \rangle - \langle c, u \rangle > \langle x, u \rangle$ and therefore, by the definition of $P^u$, we have that $x -c \not \in P$.

    Conversely, suppose that $x \in P$ is not contained in a bounded face. We want to show that there exists some direction $c \in \recc(P) \setminus \{0\}$ such that $x-c \in P$. Since $x$ is not contained in a bounded face, it is contained in the relative interior of an unbounded face $F$ (where possibly $F=P$). Since the face is unbounded, it contains a ray $x + \R_{\ge 0} c$ for some direction $c \in \recc(P)$. On the other hand, since $x \in \operatorname{int}(F)$, we have that $x - \varepsilon c \in F$ for $\varepsilon > 0$ small enough. As $\recc(P)$ is a cone, we have that $\varepsilon c \in \recc(P)$, which finishes the proof.
\end{proof}
Having this at hand, we can prove \Cref{th:decomp-reduced}.
\begin{proof}[Proof of \Cref{th:decomp-reduced}]
    Since $\pi$ induces a bijection between $\decomp{\complex}{f}$ and its image, this is also a bijection between bounded faces. By \Cref{prop:shifted-cones}, $\pi(\decomp{\complex}{f})$ is a polyhedron with recession cone $\VCP$. \Cref{th:bounded-faces} implies that $g$ is contained in a bounded face if and only if there exists no convex function $\phi \in \VCP \setminus \{0\}$ such that $g - \phi \in \pi(\decomp{\complex}{f})$. Therefore, $\pi^{-1}(g) = (g,h)$ with $ h = g-f$ is contained in a bounded face of  $\decomp{\complex}{f}$ if and only if there is no $\phi \in \VCP \setminus \{0\}$ such that $(g - \phi, g-f-\phi) = (g - \phi, h-\phi) \in \decomp{\complex}{f}$. Since $(g-\phi) - (h- \phi) = f$, this is equivalent to $g - \phi$ or $h -\phi$ being nonconvex, i.e., $(g,h)$ is reduced.
\end{proof}

Next, we describe the face relations of decomposition polyhedra by coarsenings of the corresponding functions.

\begin{definition}
We call a convex function $g \in \VCP$ a \emph{coarsening} of another convex function $g' \in \VCP$ if the unique coarsest polyhedral complex $\complex_g$ of $g$ is a coarsening of the unique coarsest polyhedral complex $\complex_{g'}$. The coarsening is called \emph{non-trivial} if $\complex_g \neq \complex_{g'}$.

For a pair of convex CPWL functions $(g,h)$, we call $(g',h')$ a coarsening of $(g,h)$ if $g-h=g'-h'$ and $g'$ and $h'$ are coarsenings of $g$ and $h$ respectively. The coarsening is called non-trivial if $(\complex_g,\complex_h) \neq (\complex_{g'},\complex_{h'})$. For a function $f \in \VC$, let $\supp_\complex(f)= \{\sigma \in \complex^{d-1} \mid w_f(\sigma) \neq 0\}$.
\end{definition}

Note that the entire definition does not depend on the underlying complex $\complex$ and the corresponding considered set $\VCP$.

\begin{lemma}
\label{lem:unique_coarsest}
   A convex function $g' \in \VCP$ is a coarsening of $g\in \VCP$ if and only if $\supp_\complex({g'}) \subseteq \supp_\complex(g)$. The coarsening is non-trivial if and only if $\supp_\complex({g'}) \subset \supp_\complex(g)$.
\end{lemma}
\begin{proof}
    First note that $B(g) \coloneqq \bigcup\limits_{\sigma \in \supp_\complex(g)}\sigma$ 
    are exactly the points where $g$ is not affine linear. Hence, the closures of the connected components of the complement of $B(g)$ are the maximal polyhedra of the unique coarsest polyhedral complex $\complex_g$ compatible with $g$.

    Let $\supp_\complex({g'}) \subseteq \supp_\complex(g)$. Equivalently, for the complement holds $(\R^d \setminus B(g)) \subseteq (\R^d \setminus B(g'))$, and the same holds for the closures of the (open) connected components, i.e., the maximal faces in $\complex^d_{g}$ and $\complex^d_{g'}$. In other words, this is equivalent to that for every face $P \in \complex_g^d$ there exists some $P' \in \complex_{g'}^d$ such that $P\subseteq P'$. Thus, $\supp_\complex({g'}) \subseteq \supp_\complex(g)$ if and only if $\complex_{g'}$ is a coarsening of $\complex_g$.

    The coarsening is non-trivial if and only if there is a $P' \in \complex_{g'}^d$ such that there is no $P \in \complex_{g}^d$ with $P' \subseteq P$. This is the case if and only if there is a $\sigma \in \complex_{g}^{d-1}$ that intersects the interior of $P'$, which occurs if and only if $\sigma \in \supp_\complex(g) \setminus \supp_\complex(g')$.
\end{proof}

We now prove the following proposition that relates coarsenings of the decompositions to inclusion relations of the minimal faces that contain the decompositions.
\begin{proposition}
\label{prop:face_relation}
   For $(g,h) \in \decomp{\complex}{f}$, let $F$ be the minimal
   face of $\decomp{\complex}{f}$ containing $(g,h)$. 
   Then $(g',h')$ is a coarsening of $(g,h)$ if and only if there is a face $G$ of $\decomp{\complex}{f}$ with $G \subseteq F$ such that $(g',h') \in G$. The coarsening is non-trivial if and only if $G \subset F$.
\end{proposition}
\begin{proof} For a face $F$, let $\mathcal{G}_F= \{\sigma \in \complex^{d-1} \mid w_g(\sigma)=0$ for all $(w_g,w_h) \in F\}$ and $\mathcal{H}_F= \{\sigma \in \complex^{d-1} \mid w_h(\sigma)=0$ for all $(w_g,w_h) \in F\}$ be the set of facets where the corresponding inequalities ensuring convexity of the functions $g$ and $h$ are tight.  It is not hard to see that $G \subseteq F$ if and only if $\mathcal{G}_F \subseteq \mathcal{G}_G$ and $\mathcal{H}_F \subseteq \mathcal{H}_G$.
    In other words, if $(g',h')$ is contained in a face $G \subseteq F$, then one can move from $(g,h)$ to $(g',h')$ without losing tight inequalities. Hence, Lemma~\ref{lem:unique_coarsest} implies that  $(g',h')$ is a coarsening of $(g,h)$. If $G \subset F$, then either $\mathcal{G}_F \subset \mathcal{G}_G$ or $\mathcal{H}_F \subset \mathcal{H}_G$. Thus, another inequality becomes tight when moving from $(g,h)$ to $(g',h')$ implying that the coarsening is non-trivial.
    
    For the converse direction, let $(g',h')$ be a coarsening of $(g,h)$, which in particular means that $g'$ and $h'$ are compatible with $\complex$. Hence, $f=g'-h'$ implies that $(g',h')\in\decomp{\complex}{f}$. Now, assume that there is no face $G\subseteq F$ such that $(g',h') \in G$. Then the line between $(g,h)$ and $(g',h')$ is not contained in $F$. Thus, a tight inequality gets lost when moving from $(g,h)$ towards $(g',h')$. Hence, without loss of generality, there is a $\sigma \in \supp_\complex(g')\setminus \supp_\complex(g)$, which according to Lemma~\ref{lem:unique_coarsest} is a contradiction to  $(g',h')$ being a coarsening of $(g,h)$. 
\end{proof}

\begin{theorem}
\label{thm:vertex_characterization}
    Let  $(g,h)\in\decomp{\complex}{f}$, then the following three statements are equivalent:
    \begin{enumerate}[itemsep=-1mm]
        \item \label{thm:vertex_characterization1}There is no non-trivial coarsening of $(g,h)$.
        \item \label{thm:vertex_characterization2} $(g,h)$ is a vertex of $\decomp{\complex}{f}$.
        \item \label{thm:vertex_characterization3} $(g,h)$ is a vertex of $\decomp{\mathcal{Q}}{f}$ for all polyhedral complexes $\mathcal{Q}$ compatible with $g$ and $h$.
    \end{enumerate}
\end{theorem}

\begin{proof}
    \ref{thm:vertex_characterization1} and \ref{thm:vertex_characterization2} are equivalent by Proposition~\ref{prop:face_relation}. \ref{thm:vertex_characterization3} trivially implies \ref{thm:vertex_characterization2}. Hence, it remains to show \ref{thm:vertex_characterization1} $\implies$ \ref{thm:vertex_characterization3}. Assume that there is a polyhedral complex $\mathcal{Q}$ compatible with $g$ and $h$ such that $(g,h)$ is not a vertex of $\decomp{\mathcal{Q}}{f}$. Then there is a vertex $(g',h')$ of $\decomp{\mathcal{Q}}{f}$ contained in the face containing $(g,h)$. By Proposition~\ref{prop:face_relation}, it follows that $(g',h')$ is a non-trivial coarsening of $(g,h)$.
\end{proof}

\begin{definition}\label{def:mini}
   A decomposition $(g,h)\in\decomp{\complex}{f}$ is called \emph{minimal},
    if it is not \emph{dominated} by any other decomposition, that is, if there is no other decomposition $(g',h')\in\decomp{\complex}{f}$ where $g'$ has at most as many pieces as $g$, $h'$ has at most as many pieces as $h$, and one of the two has strictly fewer pieces. See \Cref{fig:minimality} for a visualization.
\end{definition}
\begin{figure}[ht]
\centering
\begin{subfigure}[t]{.47\textwidth}
    \centering
    \begin{tikzpicture}[scale=0.44]
    \draw[gray,very thin] (0,0) grid (8,8);
    \draw[->] (0,0) -- (0,9) node[above]{number of pieces of $g$};
    \draw[->] (0,0) -- (9,0);
    \node at (5,-1) {number of pieces of $h$};
    \draw[red] (0,0) rectangle (3,4);
    \fill[red,opacity=0.1] (0,0) rectangle (3,4);
    \filldraw[blue] (3,4) circle (6pt);
    \filldraw[black] (5,2) circle (4pt);
    \filldraw[black] (4,2) circle (4pt);
    \end{tikzpicture}
    \caption{The blue point corresponds to a minimal decomposition.}
        \end{subfigure}
    \hspace{1em}
\begin{subfigure}[t]{.47\textwidth}
\centering
\begin{tikzpicture}[scale=0.45]
    \draw[gray,very thin] (0,0) grid (8,8);
    \draw[->] (0,0) -- (0,9) node[above]{number of pieces of $g$};
    \draw[->] (0,0) -- (9,0); \node at (5,-1) {number of pieces of $h$};
    \draw[red] (0,0) rectangle (5,2);
    \fill[red,opacity=0.1] (0,0) rectangle (5,2);
    \filldraw[black] (3,4) circle (4pt);
    \filldraw[blue] (5,2) circle (6pt);
    \filldraw[black] (4,2) circle (4pt);
    \end{tikzpicture}
    \caption{The blue point corresponds to a decomposition that is not minimal.}
\end{subfigure}
    \caption{Visualization of minimality, where a decomposition $(g,h)$ is described by the number of pieces of $g$ and $h$. A decomposition is minimal, if the rectangle spanned with $(0,0)$ does not contain another decomposition.}
    \label{fig:minimality}
    
\end{figure}

The number of pieces relates to the notion of \emph{monomial complexity} studied in \cite{tran2024minimal}, where a minimal decomposition translates to a decomposition which is minimal with respect to monomial complexity. We now give a geometric interpretation of this property in terms of $\decomp{\complex}{f}$.

\begin{theorem}\label{th:minimal-vertex}
    A minimal decomposition $(g,h)\in\decomp{\complex}{f}$ is always a vertex of $\decomp{\complex}{f}$.
\end{theorem}

\begin{proof}
    If $(g,h)$ is not a vertex, then by \Cref{thm:vertex_characterization}, there is a non-trivial coarsening $(g',h')$ of $(g,h)$. Thus, $(g,h)$ is dominated by $(g',h')$ and therefore not minimal.
\end{proof}

This theorem implies a simple finite procedure to find a minimal decomposition: enumerate all the vertices of $\decomp{\complex}{f}$ and choose one satisfying \Cref{def:mini}. It also suggests the following important special case.

\begin{proposition}\label{prop:minimal-unique-vertex}
    If $\decomp{\complex}{f}$, or equivalently $\pi(\decomp{\complex}{f})$, has a unique vertex, then this vertex corresponds to the unique minimal decomposition within $\decomp{\complex}{f}$.
\end{proposition}

\begin{proof}
    As $\decomp{\complex}{f}$ is nonempty, there must exist a minimal decomposition. By \Cref{th:minimal-vertex}, every minimal decomposition must be a vertex. As there is only one vertex, it must coincide with the unique minimal decomposition.
\end{proof}

We now demonstrate that this case is not only convenient, but also it indeed arises for important classes of functions. To this end, recall that $\pi(\decomp{\complex}{f}) = \VCP \cap (\VCP + f)$, where $\VCP$ is a convex, pointed polyhedral cone. In the following lemma, we give some sufficient conditions for such intersections of shifted cones to yield a polyhedron with a unique vertex. 
\begin{lemma}\label{lem:shifted-cones}
    Let $C \subset \R^d$ be a convex, pointed polyhedral cone, and $t \in \R^d$ 
    If $C$ is simplicial then $C \cap (C+t)$ is a translate of $C$.
    If $C$ is not simplicial, then $C \cap (C+t)$ is a translate of $C$ if $t \in C$.
\end{lemma}
\begin{proof}
    If $C$ is a simplicial full-dimensional cone, then it is the image of the nonnegative orthant under an affine isomorphism. Thus, it suffices to show that $C \cap (C + t)$ is a shifted cone for $C = \R_{\geq 0}^d$. Let $\hat{t} \in \R^d$ such that $\hat{t}_i = \max(t_i, 0)$. Then
    \[
        C \cap (C + t)
        =
        \{x \mid x_i \geq 0 \text{ and } x_i \geq t_i \}
        = 
        \{x \mid x_i \geq \hat{t}_i \}
        =
        C + \hat{t}.
    \]
    On the other hand, if $C$ is an arbitrary polyhedral cone and $t \in C$, then $C \cap (C + t) = C+t$, and hence a shifted cone.
\end{proof}
\begin{example}
    The converse of the second statement from \Cref{lem:shifted-cones} does not hold, that is, $t \not\in C$ does not imply that $C \cap (C+t)$ is not a shifted cone. Indeed, let $
    C = \cone\left( 
    \sma 1 \\ 0 \\ 0 \strix ,
    \sma 1 \\ 1 \\ 0 \strix ,
    \sma 1 \\ 0 \\ 1 \strix ,
    \sma 1 \\ 1 \\ 1 \strix 
    \right) 
    %\qquad 
    \text{ and }
    %\qquad 
    t = \sma 0 \\ 1 \\ 1 \strix.
    $
    Then $C \cap (C+t) = C + \sma 1 \\ 1 \\ 1\strix$ is a shifted cone. On the other hand, the choice $t = \sma 0 \\ 1 \\ 2 \strix$ yields the unbounded polyhedron
    $
        C \cap (C + t) = \conv\left( 
        \sma 2 \\ 2 \\ 2 \strix,
        \sma 2 \\ 1 \\ 2 \strix
        \right) + C,
    $
    which has two vertices and one line segment as bounded faces.
\end{example}
Moreover, the support of a decomposition can serve as a certificate to verify if a decomposition is a unique vertex, and hence minimal. 
For $f \in \VCmod$, let $\supp_\complex^+(f) \coloneqq \{\sigma \in \complex \mid w_f(\sigma) > 0\}$ and $\supp_\complex^-(f) \coloneqq \{\sigma \in \complex \mid w_f(\sigma) < 0\}$.
\begin{proposition}
\label{prop:unique_solution}
    If for $f \in \VCmod$, there are $g,h \in \VCP$ such that $f=g-h$ and $\supp_\complex^+(f)=\supp_\complex(g)$ as well as $\supp_\complex^-(f)=\supp_\complex(h)$, then $(g,h)$ is the unique vertex of $\decomp{\mathcal{Q}}{f}$ for every regular complete complex $\mathcal{Q}$ compatible with $f$. In this case, $g$ and $h$ have at most as many pieces as $f$.
\end{proposition}

\begin{proof}
Let $\mathcal{Q}$ be any regular complete complex that is compatible with $f$. 
Then, $g$ and $h$ are as well compatible with $\mathcal{Q}$, since $\supp_\complex(g),\supp_\complex(h) \subseteq \supp_\complex(f)$ implies that also $\supp_\mathcal{Q}(g),\supp_\mathcal{Q}(h) \subseteq \supp_\mathcal{Q}(f)$.  Let $(g',h') \in \decomp{\mathcal{Q}}{f}$.
Then it holds that $\supp^+_\mathcal{Q}(f) \subseteq \supp_\mathcal{Q}(g')$ since $w_{g'}-w_f=w_{h'} \geq 0$. Hence, $\supp_\mathcal{Q}(g) \subseteq \supp_\mathcal{Q}(g')$ and \Cref{lem:unique_coarsest} implies that $g$ is a coarsening of $g'$ and analogously it follows that $h$ is a coarsening of $h'$. Therefore, $(g,h)$ is a coarsening of every decomposition and thus by \Cref{thm:vertex_characterization} the only vertex of $\decomp{\mathcal{Q}}{f}$. Clearly, $g$ and $h$ cannot have more pieces than $f$.
\end{proof}
While \Cref{prop:unique_solution} sounds technical, it is powerful as it allows us to prove that important functions satisfy the condition of \Cref{prop:minimal-unique-vertex}.

\begin{definition}
	A \emph{hyperplane function} with $k$ hyperplanes is a function $f \colon \R^d \to \R$ given by $f(x) = \sum_{i \in [k]} \lambda_i \cdot \max\{\inner{x,a_i}+b_i,\inner{x,c_i}+d_i\}$ for any $a_i,c_i \in \R^d, b_i,d_i, \lambda_i \in \R,i \in [k]$.  
\end{definition}
Hyperplane functions are precisely the functions that are computable by a ReLU neural network with one hidden layer and appear in this context as 2-term max functions (\citet{hertrich2021towards}). They also coincide with functions computable with the hinging hyperplane model (\citet{breiman93,wang2005generalization}). 
\begin{example}[Minimal decomposition for hyperplane functions]
\label{ex:virtual_zonotopes}
    Let $f \colon \R^d \to \R$ be a hyperplane function given as $f(x) = \sum_{i \in [k]} \lambda_i \cdot \max\{\inner{x,a_i}+b_i,\inner{x,c_i}+d_i\}$. We can assume without loss of generality that the hyperplanes \[H_i= \{x \in \R^d \mid \inner{x,a_i}+b_i=\inner{x,c_i}+d_i\}\] are pairwise distinct, because otherwise we can simply adjust $\lambda_i$. The polyhedral complex $\complex$ induced by the hyperplane arrangement $\{H_i\}_{i \in [k]}$ is compatible with $f$. The convex functions $g,h$ given by  \[g(x) = \sum_{\lambda_i \geq 0} \lambda_i \cdot \max\{\inner{x,a_i}+b_i,\inner{x,c_i}+d_i\} \] and \[h(x) = \sum_{ \lambda_i < 0} -\lambda_i \cdot \max\{\inner{x,a_i}+b_i,\inner{x,c_i}+d_i\}\] are the unique minimal decomposition of $f$ since \[\supp_\complex(g)=\supp^+_\complex(f)= \{\sigma \in \complex^{d-1} \mid \sigma \subseteq \bigcup_{\lambda_i \geq 0} H_i\}\] and \[\supp_\complex(h)=\supp^-_\complex(f)= \{\sigma \in \complex^{d-1} \mid \sigma \subseteq \bigcup_{\lambda_i < 0} H_i\}.\]
    See \Cref{fig:hyperplane_minimal_decomposition} for an illustration of the minimal decomposition of a hyperplane function.
\end{example}

\begin{figure}
    \centering
        \centering
       \begin{tikzpicture}[scale=0.6]
            \node at (3,0) {$=$};
            \node at (9,0) {$-$};
           \draw[convex,very thick] (-2,-1) -- (2,1);
           \draw[convex,very thick] (-2,1) -- (2,-1);
           \draw[concave,very thick] (-2,0.5) -- (2,-0.5);
           \draw[concave,very thick] (0.5,2) -- (-0.5,-2);
           \begin{scope}[shift={(6,0)}]
               \draw[convex,very thick] (-2,-1) -- (2,1);
           \draw[convex,very thick] (-2,1) -- (2,-1); 
           \end{scope}
            \begin{scope}[shift={(12,0)}]
              \draw[concave,very thick] (-2,0.5) -- (2,-0.5);
           \draw[concave,very thick] (0.5,2) -- (-0.5,-2);
           \end{scope}
       \end{tikzpicture}
       \caption{Illustration of the unique minimal decomposition of a hyperplane function.}
       \label{fig:hyperplane_minimal_decomposition}
\end{figure}

\begin{definition}
	The \emph{$k$-th order statistic} is the function $f\colon\R^d\to\R$ that returns the $k$-th largest entry of an input vector $x\in\R^d$. For $k=\lfloor \frac{d}{2}\rfloor$, this coincides with the median.
\end{definition}

\begin{example}[Minimal decomposition of $k$-th order statistic]
\label{ex:median_unique_solution}
We construct a polyhedral complex that is compatible with the $k$-th order statistic. For $U\subseteq[d]$ with $|U|=k-1$ and $i \in [d] \setminus U$, let \[P_{i,U}= \{x \in \R^d \mid x_j \leq x_i \leq x_\ell \; \text{ for all } \ell\in U, j \in [d] \setminus U \}.\] All such polyhedra and their faces form a polyhedral complex $\complex$ that is a coarsening of the braid arrangement (see \Cref{def:braid_arrangement}) and compatible with the function $f\colon \R^d \to \R$ given by $f(x)=x_i$ for $x \in P_{i,U}$. It is not hard to see that $f=g-h$ where $g,h \in \VCP$ are convex functions given by\[g(x) \coloneqq \max\limits_{\substack{I\subseteq [d]\\ |I|= k}} \left(\sum_{i \in I}x_i\right) \text{ and } h(x) \coloneqq \max\limits_{\substack{I\subseteq [d]\\ |I|= k-1}} \left(\sum_{i \in I}x_i\right).\] Moreover, for $i,j \notin U$, let \[\sigma_{i,j,U} \coloneqq \{x \in \R^d \mid x_\ell \leq x_j = x_i \leq x_m \text{ for all } m\in U, \ell \in [d] \setminus U \}.\] Then \[\supp_\complex(g)=\supp_\complex^+(f)=\{\sigma_{i,j,U} \mid U\subseteq[d], |U|=k-2\}\] and \[\supp_\complex(h)=\supp_\complex^-(f)=\{\sigma_{i,j,U} \mid U\subseteq[d], |U|=k-1\}.\] Thus, \Cref{prop:unique_solution} implies that $(g,h)$ is the unique vertex of the decomposition polyhedra corresponding to every regular polyhedral complex compatible with $f$. See \Cref{fig:media_minimal_decomposition} for an illustration of the unique minimal decomposition of the median function.
\end{example}

\begin{figure}
    \centering
    \begin{tikzpicture}[scale=0.6]
        \draw[convex,very thick] (-6,0) -- (-6,2);
        \draw[concave, very thick] (-6,0) -- (-6,-2) ;
        \draw[convex,very thick] (-6,0) -- (-4,0);
        \draw[concave, very thick] (-6,0) -- (-8,0) ;
        \draw[convex,very thick] (-6,0) -- (-8,-2) ;
        \draw[concave, very thick] (-6,0) -- (-4,2) ;
        \node at (-3,0) {$=$};
        \node at (3,0) {$-$};
        
        \draw[convex,very thick] (0,0) -- (0,2);
        \draw[concave, very thick] (6,0) -- (6,-2) ;
        \draw[convex,very thick] (0,0) -- (2,0);
        \draw[concave, very thick] (6,0) -- (4,0) ;
        \draw[convex,very thick] (0,0) -- (-2,-2) ;
        \draw[concave, very thick] (6,0) -- (8,2);

        \node at (0.75-6,1.5) {$x_1$};
        \node at (1.5-6,0.75) {$x_2$};
        \node at (-1-6,1) {$0$};
        \node at (1-6,-1) {$0$};
        \node at (-0.75-6,-1.5) {$x_1$};
        \node at (-1.5-6,-0.75) {$x_2$};

        \node at (1.5,1) {$x_1+x_2$};
        \node at (0.5,-1) {$x_1$};
        \node at (-1,0.5) {$x_2$};

        \node at (5,-1) {$0$};
        \node at (7,0) {$x_1$};
        \node at (6,1) {$x_2$};
    \end{tikzpicture}

    \caption{Illustration of the unique minimal decomposition of the median function.}
    \label{fig:media_minimal_decomposition}
\end{figure}

\Cref{th:decomp-reduced} characterizes the reduced decompositions as bounded faces and \Cref{prop:unique_solution} provides a condition that can identify a given decomposition as the minimal one.
The natural follow-up question is how to find these decompositions. In the following, we show that this can be done via linear programming over the decomposition polyhedron.
\begin{theorem}\label{th:minimizer-unique-vertex}
A decomposition in $\pi(\decomp{\complex}{f})
        =  \VCP \cap (f+ \VCP)$ is reduced if and only if it is the minimizer in $\pi(\decomp{\complex}{f})$ of a direction
        % $u\in (\operatorname{int}((\VPP)^\vee)$ 
        contained in the interior of the dual cone of $\VCP$.
        In particular, if $\pi(\decomp{\complex}{f})$ has a single vertex, then the unique minimizer is the unique reduced and minimal decomposition. 
    Under the isomorphism to $\WC$, the direction $u$ can be chosen as $u(\sigma)=1$ for all $\sigma \in \complex^{d-1}$.
\end{theorem}
Before proving this statement, we give a description of the dual cone $(\VCP)^\vee$.
Recall that $\VCP \cong  \bigcap_{\sigma \in \complex^{d-1}}\{w \in \WC \mid w(\sigma) \geq 0\}$, i.e., the intersection of the nonnegative orthant $\{w \colon \complex^{d-1} \to \R \mid w(\sigma) \geq 0\}$ with the linear space $\WC$. 
By duality of intersections and sums, it follows that 
$(\VCP)^\vee$ is isomorphic to the Minkowski sum of the nonnegative orthant with $\WC^\perp$. In particular, any weight function $w$ with positive weights $w(\sigma) > 0$ lies in the interior. 
\Cref{th:minimizer-unique-vertex} follows from the following more general fact about faces of polyhedra.

\begin{lemma}\label{prop:minimizer-dual-cone}
    Let $C$ be a convex, pointed polyhedral cone and $P$ a polyhedron with recession cone $C$.
    Then $u \in \operatorname{int}(C^\vee)$ is a direction in the interior of the dual cone of $C$ if and only if the face $P^u$ of $P$ which is minimized by $u$ is a bounded face.
\end{lemma}
\begin{proof}
    Let $P = C + Q$, where $Q$ is a bounded polyhedron. Then for any direction $u$ holds $P^u = C^u + Q^u$. As $C$ is a pointed cone, we have that $C^u$ is bounded if and only if $u \in \operatorname{int}(C^\vee)$. Since $Q^u$ is bounded for any direction, it follows that $P^u$ is bounded if and only if $u \in \operatorname{int}(C^\vee)$.
\end{proof}

\begin{proof}[Proof of \Cref{th:minimizer-unique-vertex}]
    $\decomp{\complex}{f} = \VCP \cap (\VCP+f)$ is a polyhedron with recession cone $\VCP$. Applying \Cref{prop:minimizer-dual-cone} yields that for any  $u \in (\operatorname{int}((\VCP)^\vee)$, every minimizer in $\pi(\decomp{\complex}{f})$ lies in a bounded face, which, by \Cref{th:decomp-reduced}, are precisely the reduced decompositions. Moreover, if $\pi(\decomp{\complex}{f})$ contains a unique vertex then  by \Cref{prop:minimal-unique-vertex} this coincides with the unique minimal decomposition.
\end{proof}

\section{Translation to Virtual Polytopes}
\label{sec:virtual}

The framework of decomposition polyhedra has a natural translation to the setting of virtual polytopes. In this section, we translate our main results to this context, providing structural insights into the problem of finding ``minimal'' representatives for a given virtual polytope.

As discussed in \Cref{sec:different_perspectives}, the vector space $\cpwl^0_d$ of positively homogeneous CPWL functions is isomorphic to the vector space $\vpolytopes_d$ of virtual polytopes. Under this isomorphism, a convex positively homogeneous CPWL function corresponds to its Newton polytope. Consequently, decomposing a function $f \in \cpwl^0_d$ into a difference of convex functions, $f=g-h$, is equivalent to finding a representative for the virtual polytope $\Newt(f)$ as a Minkowski difference of two conventional polytopes, $\Newt(f) = \Newt(g) - \Newt(h)$.

This equivalence allows us to rephrase \Cref{problem} in a purely geometric language, see also \citet{tran2024minimal}. A common measure for the complexity of a polytope is its number of vertices.  The geometric problem is thus:

\begin{problem}\label{prob:virtual_polytope_rep}
	Given a virtual polytope $V \in \vpolytopes_d$, find a representative $V = P - Q$ where $P, Q \in \polytopes_d$ such that the number of vertices of $P$ and $Q$ are minimized.
\end{problem}

Our theory provides a new perspective for this problem. Let $R$ be a polytope and $\complex$ its normal fan. We restrict our search to representatives $P-Q$ where $P$ and $Q$ are in the deformation cone of $R$, or equivalently, the normal fans of $P$ and $Q$ are coarsenings of $\complex$. This is equivalent to requiring that the corresponding CPWL functions $f_P$ and $f_Q$ are compatible with $\complex$. For a virtual polytope $V$ whose corresponding CPWL function $f_V = \Newt^{-1}(V)$ is compatible with $\complex$, we define its \emph{representative polyhedron} as
\[
    \mathcal{R}_{\complex}(V) \coloneqq \{(P,Q) \in \polytopes_d \times \polytopes_d \mid V = P-Q, P,Q \in \defcone(R) \}.
\]
The linear isomorphism $\Newt$ maps the decomposition polyhedron $\decomp{\complex}{f_V}$ bijectively to the representative polyhedron $\mathcal{R}_{\complex}(V)$. This immediately allows us to translate our main theorems.

\begin{corollary}
    The set $\mathcal{R}_{\complex}(V)$ is a polyhedron that arises as the intersection of two shifted copies of the deformation cone $\defcone(R)$. 
\end{corollary}
\begin{proof}
      This is a direct translation of \Cref{prop:shifted-cones} via the vector space isomorphism $\Newt$.
\end{proof}
\begin{definition}\label{def:mini-virtual}
   A representative $(P,Q)\in\mathcal{R}_\complex(V)$ is called \emph{minimal},
    if it is not \emph{dominated} by any other representation, that is, if there is no other representation $(P',Q')\in\mathcal{R}_\complex(V)$ where $P'$ has at most as many vertices as $P$, $Q'$ has at most as many vertices as $Q$, and one of the two has strictly fewer vertices.
\end{definition}
\begin{corollary}
    A minimal representative $(P,Q) \in \mathcal{R}_{\complex}(V)$ must be a vertex of $\mathcal{R}_{\complex}(V)$.
\end{corollary}
\begin{proof}
Follows immediately from   \Cref{th:minimal-vertex} and the fact that the number of pieces of a convex, positively homogeneous CPWL function corresponds to the number of vertices of its Newton polytope. Thus, minimality as defined in \Cref{def:mini} corresponds precisely to minimizing the number of vertices of the representative polytopes.
\end{proof}
Similarly, we can translate the notion of a reduced decomposition. A representative $(P,Q)$ is \emph{irreducible} if there is no non-trivial polytope $A$ (i.e., $A$ is not just a point) such that $P' = P-A$ and $Q' = Q-A$ are both still valid polytopes. If such an $A$ exists, $(P', Q')$ would be a geometrically simpler representative of the same virtual polytope $V$.

\begin{corollary}
    A representative $(P,Q) \in \mathcal{R}_{\complex}(V)$ is irreducible if and only if it is contained in a bounded face of the representative polyhedron $\mathcal{R}_{\complex}(V)$.
\end{corollary}
\begin{proof}
    This follows directly from \Cref{th:decomp-reduced}, as subtracting a convex function $\phi$ from both components of a decomposition $(g,h)$ corresponds to subtracting the polytope $\Newt(\phi)$ from both polytopes in the representative $(P,Q)$.
\end{proof}

These results imply that the search for minimal representatives of a virtual polytope $V$ that are in the deformation cone of a summand polytope $R$ (with normal fan $\complex$) can be performed by enumerating the vertices of its representative polyhedron $\mathcal{R}_\complex(V)$. If this polyhedron has a unique vertex, that vertex corresponds to the unique minimal representative.

\section{Analysis of Existing Decompositions}\label{sec:existing}

Constructions of decompositions of CPWL functions as a difference of two convex functions have appeared in many contexts. In this section, we relate some of these existing constructions to our framework. Moreover, we provide a counterexample to a construction which was proposed by \citet{tran2024minimal} to obtain a minimal decomposition in general dimensions.

\subsection{Hyperplane Extension and Local Maxima Decomposition}\label{sec:extension-and-lattice}
The literature contains a variety of different constructions to decompose a CPWL function.
It is worth noting, however, that these constructions usually follow one of two main themes. The first theme is to construct $(g,h)$ in a way such that they are compatible with the complex $\complex$ that arises by extending the codimension-$1$ faces of $\complex_f$ to hyperplanes, see e.g. \citet{Zalgaller2000} and \citet{schlüter2021novel}. The second theme is to exploit the properties of the \emph{lattice representation} of a CPWL function \citep{1333237}. 

Both of these themes were already illustrated by \citet{KS87}, and we briefly describe their constructions here.

\begin{construction}[Hyperplane extension]
\label{constr:hyperplane_extension}
For all convex breakpoints, the local convex functions are extended to global convex functions with breakpoints supported on a single hyperplane, and the function $g$ is defined as the sum of all these functions.
To analyze it in our framework, let $\complex$ be any polyhedral complex that is compatible with $f$ and let $w_f \in \WC$ be the weight function corresponding to $f$. For $\sigma \in \complex^{d-1}$, let $H_\sigma$ be the hyperplane spanned by $\sigma$ and $\mathcal{A}^+_f=\{H_\sigma \mid w_f(\sigma) > 0\}$ be the hyperplane arrangement consisting of the hyperplanes supporting the breakpoints where $f$ is convex. Let $\mathcal{H}^+_f$ be the common refinement of the polyhedral complex induced by $\mathcal{A}^+_f$ and $\complex$. The weight function $w_g \colon \complex^{d-1} \to \R$ given by \[w_g(\sigma) \coloneqq \sum\limits_{\substack{\sigma \subseteq H_{\sigma'},\\w_f(\sigma')>0}}w_f(\sigma')\] is in $\mathcal{W}_{\mathcal{H}^+_f}$ and nonnegative and hence the corresponding function $g \in \mathcal{V}_{\mathcal{H}^+_f}$ is convex. It follows that $h \coloneqq g-f$ is convex as well, yielding the desired decomposition. 
\end{construction}

\begin{construction}[Local Maxima Decomposition]
\label{constr:lattice_representation}
Let $\{P_1,\ldots,P_m\}=\complex^d$ and $f_i$ be the unique linear extension of $f_{|P_i}$. Moreover, let $M_i \coloneqq \{j \in [m] \mid f_i(x) \geq f_j(x)$ for all $x \in P_i\}$
 and $g_i \coloneqq \max_{j \in M_i}f_j$. Then \[f = \min_{i \in [m]} \max_{j \in M_i} f_i= \min_{i \in [m]}g_i.\]  It follows that $g \coloneqq \sum_{i \in [m]}g_i$ is a convex function.  Furthermore, let $h_i \coloneqq g-g_i$, then $h \coloneqq \max_{i \in [m]}h_i$ is a convex function as well and it holds that \[g-h = g-\max_{i \in [m]} (g-g_i) = g -(g- \min_{i \in [m]} g_i) = g -(g -f)=f.\]
Let $H_{i,j} \coloneqq \{x \in \R^d \mid f_i(x) = f_j(x)\}$ and $\mathcal{A}_f =\{H_{i,j} \mid i\neq j\}$. Furthermore, let $\mathcal{H}_f$ be  the polyhedral complex induced by the hyperplane arrangement $\mathcal{A}_f$. Then we have that $g,h \in \mathcal{V}_{\mathcal{H}_f}$.
\end{construction}

In the following, we show that for the functions that compute the $k$-th order statistic, both constructions do not yield the unique minimal decompositions, which exist by \Cref{prop:minimal-unique-vertex}. This implies the following result.
\begin{proposition}\label{prop:constr-not-minimal}
    There is a CPWL function $f$ such that \cref{constr:hyperplane_extension,constr:lattice_representation} do not provide a vertex of $\decomp{\complex}{f}$ for any regular polyhedral complex $\complex$ compatible with $f$. 
\end{proposition}
\begin{example}[Hyperplane extension of $k$-th order statistic]\label{ex:hyperplane-median}
\label{ex:median_hyperplane_extension}
Let $f$ be the function from \Cref{ex:median_unique_solution}.
For any $i,j \in [d]$ and $ U \subseteq [d]$ with $i,j\in U$ and $|U|=k+1$ it holds that $\sigma_{i,j,U} \in \supp^+(f)$ and $H_{\sigma_{i,j,U}} = \{x \in \R^d \mid x_i = x_j\}$. Hence, $\complex_g$ is the braid fan (see \Cref{def:braid_arrangement}) and it holds that $g(x) = \binom{d}{k-1}\sum_{i \neq j} \max\{x_i,x_j\}$. Thus, the unique vertex $(g^*,h^*)$ from \Cref{ex:median_unique_solution} is clearly a non-trivial coarsening of the decomposition obtained from the hyperplane extension (since $g^*$ is a non-trivial coarsening of $g$) and hence the decomposition cannot be a vertex of $\decomp{\mathcal{Q}}{f}$ for any regular polyhedral complex $\mathcal{Q}$.

\end{example}
\begin{figure}[ht]
\centering
\begin{subfigure}[t]{.45\textwidth}
\centering
    \begin{tikzpicture}[scale=1]
        \draw[convex] (0,0) -- (0,2) node[above]{$1$};
        \draw[convex] (0,0) -- (0,-2) node[below]{$1$};
        \draw[convex] (0,0) -- (2,0) node[right]{$1$};
        \draw[convex] (0,0) -- (-2,0) node[left]{$1$};
        \draw[convex] (0,0) -- (-2,-2) node[below]{$\sqrt{2}$};
        \draw[convex] (0,0) -- (2,2) node[above]{$\sqrt{2}$};

        \node at (0.9,1.8) {\small $x_1 + 2 x_2$};
        \node at (1.9,0.7) {\small $2x_1 + x_2$};
        \node at (-1,1) {\small $2x_2$};
        \node at (1,-1) {\small $2 x_1$};
        \node at (-0.9,-1.8) {\small $x_1$};
        \node at (-1.9,-0.7) {\small $ x_2$};
    \end{tikzpicture}
      \caption{$g(x) = \max(x_1,x_2) + \max(x_1,0) + \max(x_2,0)$}
    \end{subfigure}
     \hspace{1.5em}
\begin{subfigure}[t]{.4\textwidth}
\centering
    \begin{tikzpicture}[scale=1]
        \draw[black, dotted] (0,0) -- (0,2) node[above]{$0$};
        \draw[concave] (0,0) -- (0,-2) node[below]{$2$};
        \draw[black,dotted] (0,0) -- (2,0) node[right]{$0$};
        \draw[concave] (0,0) -- (-2,0) node[left]{$2$};
        \draw[black,dotted] (0,0) -- (-2,-2) node[below]{$0$};
        \draw[concave] (0,0) -- (2,2) node[above]{$2\sqrt{2}$};

        \node at (0.75,1.5) {$2x_2$};
        \node at (1.5,0.75) {$2x_1$};
        \node at (-1,1) {$2x_2$};
        \node at (1,-1) {$2x_2$};
        \node at (-0.75,-1.5) {$0$};
        \node at (-1.5,-0.75) {$0$};
    \end{tikzpicture}
    \caption{$h(x) = g(x) - f(x)$, where $f$ is the median.}
    \end{subfigure}
      \caption{The hyperplane extension of the median (second largest number) of $0,x_1,x_2$ (i.e., $d=3$) (\Cref{ex:hyperplane-median}), which agrees with the local maxima decomposition (\Cref{ex:lattice-median}) up to a factor $2$. These representations do not agree for the median when $d>3$.}
\end{figure}

\begin{example}[Local maxima decomposition of $k$-th order statistic]\label{ex:lattice-median}
    Let $f$ be the function from \Cref{ex:median_unique_solution}. Then, for $U\subseteq[d]$ with $|U|=k-1$ and $i \in [d] \setminus U$, we have that
    $
        g_{i,U}(x)=\max\limits_{j \in [d]\setminus U}x_j.
    $
    Thus,
    \[
        g(x)= \sum_{i,U} g_{i,U}(x)
        =
        (d-k+1)\cdot
        \sum\limits_{\substack{S \subseteq [d]\\|S|=d-k+1}}
        \max_{j \in S}x_j.
    \]
    Note that $g$ has only breakpoints when two coordinates that are among the $k$ highest coordinates swap places in the ordering. More precisely, the order of the lowest $d-k$ coordinates is irrelevant for $g$. So, for any $T \subseteq [d]$ such that $|T|=d-k$ and any bijection $\pi \colon [k] \to [d] \setminus T$, let
    \[
        P_{T,\pi}\coloneqq
        \{x \in \R^d \mid x_j \leq x_{\pi(1)} \leq \ldots \leq x_{\pi(k)}
        \text{ for all } j \in T\}.
    \]
    It follows that the set of full-dimensional cones $\complex^d_g$ of the unique coarsest polyhedral complex $\complex_g$ compatible with $g$ is given as
    $
        \complex_g^d=\{P_{T,\pi}\}_{T,\pi}.
    $
    Again, the unique vertex $(g^*,h^*)$ from \Cref{ex:median_unique_solution} is clearly a non-trivial coarsening of the decomposition obtained from the lattice representation and hence the decomposition cannot be a vertex of $\decomp{\mathcal{Q}}{f}$ for any regular polyhedral complex $\mathcal{Q}$.
\end{example}
\Cref{constr:hyperplane_extension} and \Cref{constr:lattice_representation} provide decompositions and hence also regular polyhedral complexes $\mathcal{H}_f$ respectively $\mathcal{H}^+_f$ compatible with a given CPWL function $f$. The two examples show that the constructions do not necessarily yield vertices of the decomposition polyhedra $\decomp{\mathcal{H}_f}{f}$ respectively $\decomp{\mathcal{H}^+_f}{f}$. One might be tempted to search within the polyhedra $\decomp{\mathcal{H}_f}{f}$ and $\decomp{\mathcal{H}^+_f}{f}$ for the global optimal solution. However, the global optimal solution is not necessarily contained in these polyhedra, as we will show later. For this, we will use the fact that in dimension two, there is always a unique minimal decomposition due to a result from \citet{tran2024minimal}, which we will sketch in the next section.

\subsection{Balancing}\label{sec:minimal-balancings}
The duality between positively homogeneous convex CPWL functions and Newton polytopes, as described in \Cref{sec:different_perspectives}, serves as a motivation for \citet{tran2024minimal} to construct minimal decompositions $f = g-h$ of positively homogeneous CPWL functions as the difference of two convex such functions in dimension $2$.
\subsubsection{Balancing in dimension 2}
In this section, we present their minimal decomposition which they accomplish by introducing a single new $1$-dimensional face to $\complex_f$ and an adapted weight function to satisfy the balancing condition.

\begin{theorem}[\cite{tran2024minimal}]
\label{thm:unique_2d_balancing}
   For every positively homogeneous CPWL-function $f\colon \R^2\to \R$, there exists a unique (up to adding a linear function) minimal representation as the difference of two convex functions $g,h$. 
\end{theorem}

The decomposition can be obtained as follows.
Let $\complex_f$ be a $2$-dimensional polyhedral fan compatible with $f$ with rays $\rho_1,\dots,\rho_m \subset \R^2$ and ray generators $r_1,\dots,r_m \in \R^2$ such that $\|r_i\|=1$ for $i=1,\dots,m$. Furthermore, let $w_f$ be the corresponding element in $\mathcal W_{\complex_f}$ and $w^+_f \coloneqq \max \{w_f,0\}$. We now define an additional ray $\rho_{m+1}$ with ray generator \[r_{m+1} = -\sum_{i = 1}^m \max(w_f(\rho_i),0) r_i \] and a convex function $g$ through the weights
\[
    w_g(\rho_i) = \begin{cases}
        w_f(\rho_i) & \text{if } w_f(\rho_i) > 0, i \in [m] \\
        0 & \text{if } w_f(\rho_i) \leq 0, i \in [m] \\
         \sum_{i = 1}^m \max(w_f(\rho_i),0) & \text{if } i = m+1.
    \end{cases}
\]
This defines the convex functions $g$, $h = g-f$, and results in a minimal decomposition $f = g-h$ in the 2-dimensional positively homogeneous case. 
Considering this construction through to the duality to Newton polytopes, we can identify rays of $\complex_f$ which correspond to convex breakpoints of $f$ with edges of the Newton polytope $\Newt(g)$, and the construction from \Cref{thm:unique_2d_balancing} adds a "missing" edge to the Newton polygon $\Newt(g)$. 

\subsubsection{Fixing complex can be a problem}
With \Cref{thm:unique_2d_balancing}, we can prove that the decomposition polyhedra $\decomp{\mathcal{H}_f}{f}$ and $\decomp{\mathcal{H}^+_f}{f}$ corresponding to the polyhedral complexes obtained from hyperplane extensions $\mathcal{H}^+_f$ (\Cref{constr:hyperplane_extension}) and local maxima decompositions $\mathcal{H}^+_f$ (\Cref{constr:lattice_representation}) do not necessarily contain the global minimal decomposition of $f$ into convex functions.
\begin{proposition}
\label{re:hyperplane_extension_not_optimal}
    There is a CPWL function $f$ and convex CPWL functions $g,h$ with $f=g-h$ such that every decomposition $(g',h') \in \decomp{\mathcal{H}_f}{f}$ as well as every decomposition $(g',h') \in \decomp{\mathcal{H}^+_f}{f}$ is dominated by $(g,h)$.
\end{proposition}

\begin{proof}
    Let $\complex$ be the polyhedral complex in $\R^2$ with rays $\rho_1=\cone((1,0))$, $\rho_2=\cone((0,1))$, $\rho_3=\cone((1,2))$ and $\rho_4=\cone((2,1))$. Let $w_f(\rho_1)=w_f(\rho_2)=1$ and $w_f(\rho_3)=w_f(\rho_4)=\frac{\sqrt{5}}{3}$. Then according to \Cref{thm:unique_2d_balancing} the unique minimal decomposition is given by the complex obtained by adding the ray $\rho_5=\cone((-1,-1))$ and the weight function
    \[w_g(\rho) = 
\begin{cases}
     w_f(\rho) & \rho \in \supp_\complex^+(f) \\
     0 & \rho \in \supp_\complex^-(f) \\
     \sqrt{2} & \rho =\rho_5  
\end{cases}\] as well as
$w_h=w_g-w_f$. Nevertheless, the ray $\rho_5$ is not contained in (the support of) $\mathcal{H}_f$ and hence this solution is not in $\decomp{\mathcal{H}_f}{f}$. Since $(g,h)$ is the unique (up to adding a linear function) minimal decomposition, it follows that any solution in $\decomp{\mathcal{H}_f}{f}$ must be dominated by $(g,h)$.
Since $\mathcal{H}^+_f$ is a coarsening of $\mathcal{H}_f$, it holds that every decomposition in $\decomp{\mathcal{H}^+_f}{f}$ is contained in $ \decomp{\mathcal{H}_f}{f}$ as well, implying the result for $\decomp{\mathcal{H}^+_f}{f}$.
\end{proof}

\subsubsection{Balancing in higher dimensions}

Based on their $2$-dimensional construction, \citet{tran2024minimal} propose a procedure to reduce the $d$-dimensional case to $2$-dimensions via projections.

\begin{construction}[{\citet[Section 4.1]{tran2024minimal}}]\label{constr:ngoc}
Let $f \colon \R^d \to \R$ be a positively homogeneous CPWL-function and $\complex$ a polyhedral fan compatible with $f$. The attempt is to balance $w^+_f$ locally around every $\tau \in \complex^{d-2}$ and then ``glue together'' the local balancings to a global balancing. So, for some $\tau \in \complex^{d-2}$, suppose that $\{\sigma_1,\ldots, \sigma_k \}= \str_\complex(\tau)$ are the cones containing $\tau$. The rays spanned by $e_{\sigma_i/\tau}$ that inherit the weights $w^+_f(\sigma_i)$ for $i\in [k]$ induce a $2$-dimensional fan $\complex_\tau$ in the $2$-dimensional linear space $\spn(\tau)^\perp$ orthogonal to $\spn(\tau)$. Let $P_\tau$ be the polygon in $\spn(\tau)^\perp$ corresponding to the minimal balancing of $w^+_f$ regarded as map $w^+_f \colon \complex_\tau^1 \to \R$. Now proceed with the following steps.

\begin{enumerate}[topsep=1mm,itemsep=0mm]
    \item For every $\tau \in \complex^{d-2}$, construct the polygon $P_\tau$.
    \item Place the polygons $P_\tau$ in $\R^d$ in such a way, that whenever $\tau_1,\tau_2 \in \complex^{d-2}$ are faces of $\sigma \in \supp^+_\complex(f)$, then the edges in $P_{\tau_1}$ and $P_{\tau_2}$ that correspond to $\sigma$ are identified with each other. 
    \item Take the convex hull $P_g$ of the polygons $\{P_\tau\}_{\tau \in \complex^{d-2}}$.
    \item The support function $g$ of the polytope $P_g$ and $h\coloneqq g-f$ are a decomposition of $f$.
\end{enumerate}
\end{construction}

 One can check that for some $\sigma \in \complex^{d-1}$ and $\tau\in \complex^{d-2}$ being a face of $\sigma$, the edge $e_{\sigma}$ of length $w(\sigma)$ which is perpendicular in $\spn(\tau)^\perp$ to $\tau$ is independent of the choice of the face $\tau$. In particular, the direction of the edge $e_\sigma$ is normal to the hyperplane spanned by $\sigma$.  However, it remained unclear whether or not the second step in this procedure is always well-defined, that is, that placing the polygons in such a coherent way is possible. To make this more precise, let for some $\tau \in \complex^{d-2}$ the edges of the polygon $P_\tau$ be given in a cyclic way $\{e_{\sigma_1}, \ldots, e_{\sigma_m}\}$. Placing a polygon $P_\tau$ refers to choosing an $x_\tau \in \R^d$ and defining the placed polygon as $P_\tau(x_\tau)=\conv(x_\tau,x_\tau+e_{\sigma_1},x_\tau+ e_{\sigma_1}+e_{\sigma_2},\ldots,x_\tau+\sum_{i=1}^m e_{\sigma_m})$. Placing them in a coherent way means choosing an $x_\tau \in \R^d$ for every $\tau \in \complex^{d-2}$ such that $P_{\tau_1}(x_{\tau_1}) \cap P_{\tau_2}(x_{\tau_2}) = \conv(x_\sigma, x_\sigma + e_\sigma)$ for some $x_\sigma \in \R^d$ whenever $\tau_1$ and $\tau_2$ are faces of $\sigma$. A priori it is not clear that such $x_\tau$ always exist. The following example will in fact show that the resulting linear equation system not always yields a solution.

\begin{example}[Counterexample to \Cref{constr:ngoc}]

\label{ex:counter_example_balancing}
   \Cref{fig:non_coherent_polygons} is an illustration of $4$ polygons with labelled edges that cannot be placed in $\R^3$ such that the edges of different polygons with the same label are identified with each other. Hence, applying the above procedure to the CPWL-function $f\colon \R^3 \to \R$ given by 
    \[
    f(x)= \max\{0,\max\limits_{\substack{i,j\in [3]\\ i\neq j}}\{\min \{x_i,x_j-x_i\}\}\}
    \] is not well-defined since these $4$ polygons arise and should be identified in the indicated way, which is impossible.

\begin{figure}[ht]
    \centering
\begin{tikzpicture}
    \draw(0,0,0) -- ++(0,2,-2) node[midway,left]{$e_{\sigma_1}$};
    \draw[blue] (0,2,-2) --++(1,-1,1) node[midway,right]{$e_{\sigma_2}$};
    \draw[orange](1,1,-1) -- (0,0,0) node[midway,right]{$e_{\sigma_3}$};
    \node at (0.5,1,-1) {$P_{\rho_1}$};
    
    \draw[blue] (1.5,1.5,-2) -- ++(-1,1,-1) node[midway,left]{$e_{\sigma_2}$};
    \draw (0.5,2.5,-3) -- ++(1,0,-1) node[midway,above]{$e_{\sigma_7}$};
    \draw (1.5,2.5,-4) -- ++(1,-1,1) node[midway,right]{$e_{\sigma_8}$};
    \draw[green] (2.5,1.5,-3) -- (1.5,1.5,-2) node[midway,below]{$e_{\sigma_9}$};
    \node at (1.5,2,-3) {$P_{\rho_3}$};
    
    \draw (3.5,0.5,-2) -- ++(1,-1,0) node[midway,right]{$e_{\sigma_{10}}$};
    \draw (4.5,-0.5,-2) -- ++(0,-1,1) node[midway,right]{$e_{\sigma_{11}}$};
    \draw (4.5,-1.5,-1) -- ++(-1,0,1) node[midway,below]{$e_{\sigma_{12}}$};
    \draw[purple] (3.5,-1.5,0) -- ++(-1,1,0)  node[midway,below]{$e_{\sigma_{4}}$};
    \draw (2.5,-0.5,0) -- ++(0,1,-1)  node[midway,left]{$e_{\sigma_{13}}$};
    \draw[green] (2.5,0.5,-1) -- (3.5,0.5,-2)  node[midway,above]{$e_{\sigma_9}$};
    \node at (3.5,-0.5,-1) {$P_{\rho_4}$};

    \draw (1,-4,-1) -- ++(1,1,-1) node[midway,right]{$e_{\sigma_{5}}$};
    \draw[purple] (2,-3,-2) -- ++(-1,1,0) node[midway,right]{$e_{\sigma_{4}}$};
    \draw[orange] (1,-2,-2) -- ++(-1,-1,1) node[midway,left]{$e_{\sigma_{3}}$};
    \draw (0,-3,-1) -- (1,-4,-1) node[midway,below]{$e_{\sigma_{6}}$};
    \node at (1,-3,-1.5) {$P_{\rho_2}$};
\end{tikzpicture}
\caption{4 polygons that cannot be placed in a coherent way in $\R^3$}
\label{fig:non_coherent_polygons}
\end{figure}

         We describe the $2$-skeleton of a polyhedral fan $\complex$ that is compatible with $f$. Since $f$ is positively homogeneous,  $\complex$ is a fan. Thus, its combinatorics can be visualized by intersecting it with a sphere centered at the origin and projecting radially; this is the role of \Cref{fig:fan_counter_example}. The full-dimensional cones of the fan are the linear regions of $f$, and on each such region $f$ agrees with one of the linear forms appearing in its defining expression, namely one of $0$, $x_i$, or $x_j-x_i$.Recall that $e_i$ is the $i$-th standard unit vector.
    The rays are given as follows:
    \begin{multline*}
        \complex^1= \{\cone(-e_i),
        \cone(e_i),
        \cone(e_i+e_j),
        \cone(e_i+e_j+2e_k),
        \cone(e_i+2e_j+2e_k),\\
        \cone(e_i+e_j+2e_k),
        \cone(e_i+e_j+e_k)\mid i,j,k \in [3] \text{ pairwise distinct}\}
    \end{multline*}
    % \begin{align*}
    %     \complex^1= \{&\cone(-e_i),
    %     \cone(e_i),
    %     \cone(e_i+e_j),
    %     \cone(e_i+e_j+2e_k),
    %     \cone(e_i+2e_j+2e_k),\\&
    %     \cone(e_i+e_j+2e_k),
    %     \cone(e_i+e_j+e_k)\mid i,j,k \in [3] \text{ pairwise distinct}\}
    % \end{align*}
    
    and the $2$-dimensional cones as
    \begin{align*}
        \complex^2= \{&\cone(e_i,-e_j),
        \cone(-e_i,e_i+e_j+2e_k),
        \cone(-e_i,e_j+e_k),
        \cone(e_i,e_i+e_j+2e_k),\\&
        \cone(e_j+e_k,e_i+e_j+2e_k),
        \cone(e_i+e_j+2e_k,e_i+2e_j+2e_k),\\&
        \cone(e_i+e_j+2e_k,e_i+e_j+e_k),
        \cone(e_i+2e_j+2e_k,e_i+e_j+e_k),\\&
        \cone(e_j+e_k,e_i+2e_j+2e_k) \mid i,j,k \in [3] \text{ pairwise distinct}\}
    \end{align*}
    The weight function $w_f \colon \complex^2\to \R$ given by  \begin{align*}
        &w(\cone(e_i,-e_j))=1\\
        &w(\cone(-e_i,e_i+e_j+2e_k))=-\sqrt{5}\\
        &w(\cone(-e_i,e_j+e_k))=\sqrt{8}\\
        &w(\cone(e_i,e_i+e_j+2e_k))=\sqrt{2},\\
        &w(\cone(e_j+e_k,e_i+e_j+2e_k))=\sqrt{3},\\
        &w(\cone(e_i+e_j+2e_k,e_i+2e_j+2e_k))=-\sqrt{5},\\
        &w(\cone(e_i+e_j+2e_k,e_i+e_j+e_k))=\sqrt{2},\\
        &w(\cone(e_i+2e_j+2e_k,e_i+e_j+e_k))=\sqrt{2},\\
        &w(\cone(e_j+e_k,e_i+2e_j+2e_k))=0
    \end{align*}
corresponds to the function $f$ and is therefore balanced. See \Cref{fig:fan_counter_example} for a $2$-dimensional illustration of $\complex$.
\begin{figure}

    \centering
\begin{tikzpicture}[scale=0.8]
\node (-100) at (-9,0) {$(-1,0,0)$};
\node (010) at (0,9) {$(0,1,0)$};
\node (00-1) at (9,9) {$(0,0,-1)$};
\node (100) at (9,0) {$(1,0,0)$};
\node (0-10) at (0,-9) {$(0,-1,0)$};
\node (001) at (-9,-9) {$(0,0,1)$};

\node (011) at (-6,0){$\rho_1=(0,1,1)$};
\node (211) at (3,0){$(2,1,1)$};
\node (110) at (6,6){$(1,1,0)$};
\node (112) at (-3,-3){$\rho_2=(1,1,2)$};
\node (121) at (0,3){$\rho_3=(1,2,1)$};
\node (101) at (0,-6){$(1,0,1)$};

\node (122) at (-3,0){$(1,2,2)$};
\node (221) at (3,3){$(2,2,1)$};
\node (212) at (0,-3){$(2,1,2)$};

\node (00) at (0,0){$\rho_4=(1,1,1)$};

\draw[convex] (-100) -- (010);% node[midway,left]{$1$};
\draw[convex] (010) -- (00-1);% node[midway,above]{$1$};
\draw[convex] (00-1) -- (100);% node[midway,right]{$1$};
\draw[convex] (100) -- (0-10);% node[midway,right]{$1$};
\draw[convex] (0-10) -- (001);% node[midway,below]{$1$};
\draw[convex] (001) -- (-100);% node[midway,left]{$1$};

\draw[convex] (-100) -- (011) node[midway,below]{$\sigma_1$};
\draw[convex] (100) -- (211);% node[midway,above]{$\sqrt{2}$};
\draw[convex] (001) -- (112) node[midway,left]{$\sigma_6$};
\draw[convex] (00-1) -- (110);% node[midway,right]{$2$};
\draw[convex] (0-10) -- (101);% node[midway,left]{$2$};
\draw[convex] (010) -- (121) node[midway,right]{$\sigma_7$};

\draw[convex] (011) -- (121) node[midway,below]{$\sigma_2$};
\draw[convex] (121) -- (110) node[midway,below]{$\sigma_8$};
\draw[convex] (110) -- (211);% node[midway,left]{$\sqrt{3}$};
\draw[convex] (211) -- (101);% node[midway,left]{$\sqrt{3}$};
\draw[convex] (101) -- (112) node[midway,above]{$\sigma_5$};
\draw[convex] (112) -- (011) node[midway,right]{$\sigma_3$};

\draw[concave,dashed]  (122) -- (121);% node[midway,right]{$-\sqrt{5}$};
\draw[concave,dashed]  (122) -- (112);% node[midway,right]{$-\sqrt{5}$};
\draw[concave,dashed]  (212) -- (112);% node[midway,above]{$-\sqrt{5}$};
\draw[concave,dashed]  (212) -- (211);% node[midway,left]{$-\sqrt{5}$};
\draw[concave,dashed]  (221) -- (211);% node[midway,left]{$-\sqrt{5}$};
\draw[concave,dashed]  (221) -- (121);% node[midway,below]{$-\sqrt{5}$};

\draw[convex] (00) -- (212) node[midway,left]{$\sigma_{12}$};
\draw[convex] (00) -- (121) node[midway,left]{$\sigma_9$};
\draw[convex] (00) -- (122) node[midway,above]{$\sigma_{13}$};
\draw[convex] (00) -- (211) node[midway,above]{$\sigma_{11}$};
\draw[convex] (00) -- (221) node[midway,left]{$\sigma_{10}$};
\draw[convex] (00) -- (112) node[midway,left]{$\sigma_4$};

\draw[concave,dashed] (0-10) -- (112);% node[midway,left]{$-\sqrt{5}$};
\draw[concave,dashed] (-100) -- (112);% node[midway,below]{$-\sqrt{5}$};
\draw[concave,dashed] (-100) -- (121);% node[midway,above]{$-\sqrt{5}$};
\draw[concave,dashed] (0-10) -- (211);% node[midway,right]{$-\sqrt{5}$};
\draw[concave,dashed] (00-1) -- (211);% node[midway,right]{$-\sqrt{5}$};
\draw[concave,dashed] (00-1) -- (121);% node[midway,above]{$-\sqrt{5}$};

\draw[gray,dotted,very thick] (101) -- (212);% node[midway,left]{$0$};
\draw[gray,dotted, very thick] (011) -- (122);% node[midway,below]{$0$};
\draw[gray,dotted,very thick] (110) -- (221);% node[midway,right]{$0$};
\end{tikzpicture}
\caption{A $2$-dimensional representation of $\complex$. The blue lines correspond to convex breakpoints of the function $f$, that is, a cone $\sigma \in \complex^2$ such that $w(\sigma) > 0$. The concave breakpoints ($w(\sigma)<0)$ are dashed and colored in green. $f$ has no breakpoints on the gray, dotted lines ($w(\sigma)=0$).}
\label{fig:fan_counter_example}
\end{figure}

Consider the $4$ rays 
\begin{align*}
    \rho_1 = \cone((0,1,1)),\rho_2 = \cone((1,1,2)),\rho_3 = \cone((1,2,1)),\rho_4=\cone((1,1,1))
\end{align*} For the four rays $\rho_1,\rho_2,\rho_3,\rho_4$ considered below, the positive part $w_f^+$ is already balanced in the corresponding two-dimensional quotient fans. Equivalently, the positive edge vectors around each of these rays already close up to form a polygon. Hence, in the local two-dimensional step of \Cref{constr:ngoc}, no additional ``missing'' edge is introduced at these rays. The obstruction is therefore not local balancing, but the subsequent global gluing step: the locally constructed polygons cannot be placed in $\R^3$ so that equally labelled edges are identified coherently. We will see that the corresponding polygons $P_{\rho_1},P_{\rho_2},P_{\rho_3}$ and $P_{\rho_4}$ equal the ones in Figure~\ref{fig:non_coherent_polygons}. The $2$-dimensional cones which are in the stars of the $4$ rays are the following:
\begin{align*}
    &\sigma_{1}=\cone((0,1,1),(-1,0,0)),\, \sigma_2=\cone((0,1,1),(1,2,1)),\,  \sigma_3=\cone((0,1,1),(1,1,2)),\\&\sigma_4=\cone((1,1,2),(1,1,1)), \sigma_5=\cone((1,1,2),(1,0,1)),\,  \sigma_6=\cone((1,1,2),(0,0,1)),\\&\sigma_7=\cone((1,2,1),(0,1,0)),\,  \sigma_8=\cone((1,2,1),(1,1,0)),\, \sigma_9=\cone((1,2,1),(1,1,1)),\\&\sigma_{10}=\cone((1,1,1),(2,2,1)),\,  \sigma_{11}=\cone((1,1,1),(2,1,1),\,  \sigma_{12}=\cone((1,1,1),(2,1,2),\\& \sigma_{13}=\cone((1,2,2),(1,1,1)),
\end{align*}

The directions of the edges of the polygons are given by the normal vectors of the hyperplanes spanned by the corresponding $2$-dimensional cone and their length by the weight of the corresponding $2$-dimensional cone.

\begin{align*}
&e_{\sigma_1}=(0,2,-2),\,  e_{\sigma_2}=(1,-1,1),\,  e_{\sigma_3}=(-1,-1,1), \,e_{\sigma_4}=(1,-1,0),\,  e_{\sigma_5}=(1,1,-1),\,  \\ 
&e_{\sigma_6}=(-1,1,0),\, e_{\sigma_7}=(1,0,-1),\,  e_{\sigma_8}=(-1,1,-1),\,  e_{\sigma_9}=(-1,0,1), \,e_{\sigma_{10}}=(1,-1,0),\\ 
&e_{\sigma_{11}}=(0,-1,1),\, e_{\sigma_{12}}=(-1,0,1),\, e_{\sigma_{13}}=(0,1,-1)
\end{align*} 
In order to construct the polygons, one needs to consider the orientation of the edge (and the normal vector) in the particular polygon.
One can convince themselves that the polygons $P_{\rho_1},P_{\rho_2},P_{\rho_3}$ and $P_{\rho_4}$  of the $4$ rays are given as:
\begin{align*}
P_{\rho_1}(x_{\rho_1})&=\conv(x_{\rho_1}+e_{\sigma_1}, x_{\rho_1}+e_{\sigma_1}+e_{\sigma_2},x_{\rho_1}+e_{\sigma_1}+e_{\sigma_2}+e_{\sigma_3}) \\
P_{\rho_2}(x_{\rho_2})&=\conv(x_{\rho_2}-e_{\sigma_3}, x_{\rho_2}-e_{\sigma_3}-e_{\sigma_4},x_{\rho_2}-e_{\sigma_3}-e_{\sigma_4}-e_{\sigma_5},x_{\rho_2}-e_{\sigma_3}-e_{\sigma_4}-e_{\sigma_5}-e_{\sigma_6})
\\P_{\rho_3}(x_{\rho_3})&=\conv(x_{\rho_3}-e_{\sigma_2},x_{\rho_3}-e_{\sigma_2}+e_{\sigma_7},x_{\rho_3}-e_{\sigma_2}+e_{\sigma_7}+e_{\sigma_8},x_{\rho_3}-e_{\sigma_2}+e_{\sigma_7}+e_{\sigma_8}+e_{\sigma_9})
\\P_{\rho_4}(x_{\rho_4})&=\conv(x_{\rho_3}+e_{\sigma_9},x_{\rho_4}+e_{\sigma_9}+e_{\sigma_{10}},x_{\rho_4}+e_{\sigma_9}+e_{\sigma_{10}}+e_{\sigma_{11}},\\&x_{\rho_4}+e_{\sigma_9}+e_{\sigma_{10}}+e_{\sigma_{11}}+e_{\sigma_{12}},x_{\rho_4}+e_{\sigma_9}+e_{\sigma_{10}}+e_{\sigma_{11}}+e_{\sigma_{12}}+e_{\sigma_4},\\&x_{\rho_4}+e_{\sigma_9}+e_{\sigma_{10}}+e_{\sigma_{11}}+e_{\sigma_{12}}+e_{\sigma_4}+e_{\sigma_{13}}).
\end{align*}

Figure~\ref{fig:non_coherent_polygons} already shows that they cannot be placed in a coherent way.
To make this mathematically precise, one obtains a system of linear equations for $x_{\rho_1},x_{\rho_2},x_{\rho_3}$ and $x_{\rho_4}$, by plugging in the values for the normal vectors, that ensures that the edges for the same $2$-dimensional cone in different polygons are identified. This linear equation system does not have a solution.

  \end{example}

Our conclusion is that the construction by \citet{tran2024minimal} does not extend beyond the $2$-dimensional case, leaving it an open problem to find any finite algorithm that guarantees to return a minimal decomposition without fixing an underlying polyhedral complex. Note that for any fixed underlying polyhedral complex, \Cref{th:minimal-vertex} implies a finite algorithm by enumerating the vertices of the decomposition polyhedron, as discussed earlier.

\section{Submodular Functions}\label{sec:submodular}
In this section, we demonstrate that a special case of our framework is to decompose a general set function into a difference of submodular set functions and translate our results to this setting.
Such decompositions are a popular approach to solve optimization problems, as discussed in the introduction.

\begin{definition}
\label{def:braid_arrangement}
    The \emph{braid arrangement} in $\R^d$ is the hyperplane arrangement consisting of the $\binom{d}{2}$ hyperplanes $x_i=x_j$, with $1\leq i<j \leq d$.    The \emph{braid fan} $\B_d$ is the polyhedral fan induced by the braid arrangement.
\end{definition}

 Let $\F_d$ be the vector space of set functions from $2^{[d]}$ to $\R$. We first show that functions in $\V{\B_d}$ are in one-to-one correspondence with the set functions $\F_d$. To this end, for a set $S\subseteq[d]$, let $\ind_S=\sum_{i\in S} e_i$ be the indicator vector of $S$, that is, the vector that contains entries $1$ for indices in $S$ and $0$ otherwise.

\begin{proposition}
\label{prop:iso_cpwl_to_setfunction}
    The linear map $\Phi \colon \FB{d} \to \Sf{d}$  given by $\Phi(f) = F(S) = f(\ind_S)$ is an isomorphism.
\end{proposition}

\begin{proof}
 The map $\Phi$ is clearly a linear map.
    To prove that $\Phi$ is an isomorphism, we show that a function $f\in \FB{d}$ is uniquely determined by its values on $\{\ind_S\}_{S\subseteq [d]}$ and any choice of real values $\{y_S\}_{S\subseteq [d]}$ gives rise to a function $f\in \FB{d}$ such that $f(\ind_S)=y_S$.

    First, note that the maximal cones of $\B_d$ are of the form $C_{\pi}=\{x \in \R^d \mid x_{\pi(1)} \leq \ldots \leq  x_{\pi(d)}\}$ for a permutation $\pi \colon [d] \to [d]$. There are exactly the $d+1$ indicator vectors $\{\ind_{S_i}\}_{i=0,\ldots,d}$ contained in $C_{\pi}$, where $S_i \coloneqq\{\pi(d+1-i),\ldots,\pi(d)\}$ for $i \in [d]$ and $S_0 \coloneqq \emptyset$. Moreover, the vectors $\{\ind_{S_i}\}_{i=0,\ldots,d}$ are affinely independent and hence the values $\{f(\ind_{S_i})\}_{i=0,\ldots,d}$ uniquely determine the affine linear function $f_{|C_{\pi}}$. Therefore, $f$ is uniquely determined by $\{f(\ind_{S})\}_{S\subseteq[d]}$.
    
    Given any values $\{y_S\}_{S \subseteq [d]}$, by the discussion above, there are unique affine linear maps $f_{|C_{\pi}}$ yielding $f_{|C_{\pi}}(\ind_S)=y_S$ for all $S \subseteq [d]$ such that $\ind_S \in C_{\pi}$. It remains to show that the resulting function $f$ is well-defined on the facets. Any such facet is of the form \[C_{\pi,i}=\{x \in \R^d \mid x_{\pi(1)} \leq \ldots \leq x_{\pi(i)} = x_{\pi(i+1)} \leq \ldots \leq  x_{\pi(d)}\},\] which is the intersection of $C_{\pi}$ and $C_{\pi \circ (i,i+1)}$,  where $(i,i+1)$ denotes the transposition swapping $i$ and $i+1$. However, the indicator vectors $\{\ind_{S_i}\}_{i\in[d]\setminus \{i\}}$ contained in $C_{\pi,i}$ are a subset of the indicator vectors contained in $C_{\pi}$ and $C_{\pi \circ (i,i+1)}$. Therefore, it holds that $f_{|C_{\pi}}(x)=f_{|C_{\pi \circ (i,i+1)}}(x)$ for all $x \in C_{\pi,i}$ implying that $f$ is well-defined as a CPWL function.
\end{proof}

If we think this the other way around, starting with a set function $F$, then $f=\Phi^{-1}(F)$ is by definition a continuous extension of $F$. It turns out that this particular extension is known as the \emph{Lov{\'a}sz extension}~\citep{lovasz1983submodular}, as we argue below. The Lov{\'a}sz extension is an important concept in the theory and practice of submodular function optimization as it provides a link between \emph{discrete} submodular functions and \emph{continuous} convex functions.

\begin{definition}
    For a set function $F \colon 2^{[d]} \to \R$, the \emph{Lov{\'a}sz extension} $f \colon \R^d \to \R$ is defined by $f(x) = \sum_{i=0}^d \lambda_i F(S_i)$, where $\emptyset=S_0 \subset S_1 \subset \ldots \subset S_d = [d]$ is a chain such that $ \sum_{i=1}^d \lambda_i \ind_{S_i} =x$ and $\lambda_i \geq 0$ for all $i\in [d]$ and $\lambda_0=1-\sum_{i=1}^d\lambda_i$.
\end{definition}

\begin{remark}
    In many contexts in the literature, the Lov{\'a}sz extension is only defined on the hypercube $[0,1]^d$. For our purposes, it is more convenient to omit this restriction, which is captured by the above definition.
\end{remark}

The intuition of the Lov\'asz extension can already be seen in the proof of \Cref{prop:iso_cpwl_to_setfunction}: depending on the ordering of the components of an input vector~$x$, the Lov\'asz extension writes~$x$ as an affine combination of indicator vectors $\ind_{S_i}$, and uses the coefficients of the affine combination to compute the value $f(x)$. Following the intuition, we see in the next proposition that $\Phi^{-1}(F)$ is actually the Lov{\'a}sz extension of~$F$.

\begin{proposition}\label{prop:lovasz-extension}
    For a set function $F\in \F_d$, the function $f=\Phi^{-1}(F)$ is precisely the Lov{\'a}sz extension of~$F$.
\end{proposition}
\begin{proof}
    By the definition of the Lov{\'a}sz extension, it follows that it is compatible with $\B_d$. Thus, the Lov{\'a}sz extension is contained in $\FB{d}$. Moreover, as it is an extension that fixes indicator vectors, it follows that~$\Phi$ applied to the Lov{\'a}sz extension of $F$ gives back $F$. As $\Phi$ is an isomorphism by \Cref{prop:iso_cpwl_to_setfunction}, the Lov{\'a}sz extension must be exactly $\Phi^{-1}(F)$.
\end{proof}

\begin{definition}
   A set function $F \colon 2^{[d]} \to \R$ is called \emph{submodular} if 
   \begin{equation}
       F(A) +F(B) \geq F(A \cup B) + F(A \cap B)\label{eq:submodular}
   \end{equation}for all $A,B \subseteq [d]$. $F$ is called \emph{modular} if equality holds for all $A,B \subseteq [d]$.
\end{definition}

The following well-known property is key to the insights of this section.

\begin{proposition}[\cite{lovasz1983submodular}]
    A set function $F$ is submodular if and only if its Lov{\'a}sz extension $f=\Phi^{-1}(F)$ is convex.
\end{proposition}

Applying our insights from \Cref{sec:compatible} to the previous proposition, we obtain the following well-known statement.

\begin{corollary}
    The set of submodular functions forms a  polyhedral cone $\submod_d$ in the vector space $\F_d$.
\end{corollary}

In particular, we can specialize \Cref{problem} in the setting of this section as follows.

\begin{problem}\label{prob:decomp_lovasz_extension2}
    Given a set function $F\in\F_d$, how to decompose it into a difference of submodular set functions such that their Lov{\'a}sz extensions have as few pieces as possible?
\end{problem}

Having a Lov{\'a}sz extension with few pieces is desirable because it allows the submodular function to be stored and accessed efficiently during computational tasks. As \Cref{prob:decomp_lovasz_extension2} is a special case of \Cref{problem}, we are able to translate our results from \Cref{sec:compatible} to the setting of submodular functions.

Let $\modular_d\subseteq \F_d$ be the vector space of \emph{modular} functions, that is, set functions that satisfy \Cref{eq:submodular} with equality. Note that a set function is modular if and only if its Lov{\'a}sz extension is an affine function \citep{lovasz1983submodular}. Since for any $M  \in \modular_d$, a set function $F$ is submodular if and only if $F+M$ is submodular, we define the vector space $\overline{\F}_d = \F_d / \modular_d$ of set functions modulo modular functions. Furthermore, let $\overline{\submod}_d$ be the cone of submodular functions in this quotient.

A decomposition $(G,H) \in \overline{\submod}_d \times \overline{\submod}_d$  of a set function $F\in \overline{\F}_d$ is called \emph{irreducible} if there does not exist a submodular function $I \in \overline{\submod}_d\setminus \{0\}$ such that $G-I$ and $H-I$ are submodular. Since a set function $M \in \F_d$ is modular if and only if $\Phi^{-1}(M)$ is affine linear, the isomorphism $\Phi$ of Proposition~\ref{prop:iso_cpwl_to_setfunction} descends to an isomorphism $\Phi \colon \overline{\F}_d \to \Vmod{\B_d}$ such that $\Phi(\overline{\submod}_d) = \VPmod{\B_d}$.

    For a set function $F \in \overline{\F}_d$, the set of decompositions \[\decomp{}{F}\coloneqq \{(G,H) \in \overline{\submod}_d \times \overline{\submod}_d \mid F=G-H\}\] is a polyhedron.

\begin{corollary}
 A decomposition $(G,H)$ is irreducible if and only if $(G,H)$ is contained in a bounded face of $\decomp{}{F}$.
\end{corollary}
\begin{proof}
    The extension of $\overline{\Phi}$ to the cartesian product $\overline{\Phi} \times \overline{\Phi} \colon \Vmod{\B_d} \times \Vmod{\B_d} \to \overline{\F}_d \times \overline{\F}_d$ is an isomorphism. Then the statement follows from the fact that $\decomp{}{F}= (\overline{\Phi}\times \overline{\Phi}) (\decomp{\B_d}{\Phi^{-1}(F)}$ and \Cref{th:decomp-reduced}.
\end{proof}
\begin{definition}
    For a submodular function $F \colon 2^{[d]} \to \R$, the \emph{base polytope} $B(F)$ is defined as \[B(F) \coloneqq \{x \in \R^d \mid \sum_{i \in S} x_i \leq F(S) \, \forall S\subset [d], \sum_{i \in [d]}x_i = F([d])\}.\]
\end{definition}
 
Since we factored out modular functions, we can assume without loss of generality that a set function $F \in \overline{\F}_d$ is \emph{normalized}, that is,  $F(\emptyset) =0$.  In this case, $f=\Phi^{-1}(F)$ is positively homogeneous. For the remainder of this section, we will assume all set functions to be normalized and all CPWL functions to be positively homogeneous. If $F$ is submodular, $f$ agrees with the support function of the base polytope $B(F)$, and therefore $B(F)$ is the Newton polytope $\Newt(f)$ of the Lov\'asz extension $f$ (see e.g. \cite{Aguiar2017HopfMA} Theorem~12.3.).  The Newton polytopes of functions that differ by a linear map are translations of each other and modular functions correspond to linear functions \citep{lovasz1983submodular}.  Hence, the set of base polytopes in $\R^d$ modulo translation is precisely the set of Newton polytopes of functions compatible with the braid fan, which are also known as \emph{generalized permutahedra}. We denote the cone of generalized permutahedra by $\GPmod_d$ and hence the maps \[B \colon \overline{\F}_d \to \GPmod_d, F \mapsto B(F)\] and \[\Newt \colon \VPmod{\B_d} \to \GPmod_d, f \mapsto \Newt(f)\] are well-defined and we obtain the following diagram:
\[
\begin{tikzcd}
   \overline{\submod}_d \arrow{r}{B} \arrow{d}[swap]{\overline{\Phi}}&\GPmod_d \\
   \VPmod{\B_d} \arrow[swap]{ur}{\Newt}\end{tikzcd}\]
 In this setting, we call a decomposition $(G,H)\in\decomp{}{F}$ \emph{minimal},
    if it is not \emph{dominated} by any other decomposition, that is, if there is no other decomposition $(G',H')\in\decomp{}{F}$ where $B(G')$ has at most as many vertices as $B(G)$, $B(H')$ has at most as many vertices as $B(H)$, and one of the two has strictly fewer vertices.

For a tuple of submodular functions $(G,H) \in \overline{\submod} \times \overline{\submod}$, let $(\mathcal{P}_{G},\mathcal{P}_H)$  be the tuple of the normal fans of the base polytopes $B(G)$ and $B(H)$. 
A decomposition $(G',H')\in\decomp{}{F}$ is called a \emph{(non-trivial) coarsening} of $(G,H)\in\decomp{}{F}$ if $\mathcal{P}_{G'}$ and $\mathcal{P}_{H'}$ are coarsenings of $\mathcal{P}_{G}$ and $\mathcal{P}_H$, respectively (and at least one of them is a non-trivial coarsening). 

\begin{corollary}
   $(G,H) \in \decomp{}{F}$ is a vertex if and only if there is no non-trivial coarsening of $(G,H)$.
\end{corollary}
\begin{proof}
Since $g=\Phi^{-1}(G)$ and $h=\Phi^{-1}(H)$ are the support functions of the base polytopes $B(G)$ respectively $B(H)$, the tuple of normal fans $(\mathcal{P}_{G},\mathcal{P}_H)$ agrees with the tuple $(\PP_g,\PP_h)$ of the unique coarsest polyhedral complexes  compatible with $g$ and $h$. Hence by \Cref{thm:vertex_characterization}, there is no non-trivial coarsening of $(G,H)$ if and only if $(g,h)$ is a vertex of $\decomp{\PP}{f}$ which is the case if and only if $(G,H)$ is a vertex of $\decomp{}{F}$.
\end{proof}
\begin{corollary}
\label{cor:submod}
    For a normalized set function $F \in \overline{\F}_d$, a minimal decomposition of $F$ is a vertex of $\decomp{}{F}$.
\end{corollary}
\begin{proof}
    If $(G,H)$ is not a vertex, then there is a coarsening $(G',H') \in \decomp{}{F}$ of $(G,H)$ implying that $(G',H')$ dominates $(G,H)$. 
\end{proof}

The following example shows that the Lov\'asz extensions of cut functions are hyperplane functions thus admit a unique minimal decomposition into submodular functions, which are themselves cut functions. In particular, the Lov\'asz extensions of the decomposition have at most as many pieces as the  Lov\'asz extensions of the original cut function.

\begin{example}[Minimal decompositions of cut functions]
\label{ex:cut_function}
Let $G=(V,E)$ be a graph where $V=[n]$ and $c \colon E \to \R$ a weight function on the edges. Let $F \in \F_n$ be the cut function given by $F(S)=\sum_{\{u,v\} \in \delta(S)}c(\{u,v\})$, where $\delta(S) \coloneqq \{\{u,v\}   \in E \mid u \in S, v \in V\setminus S\}$. The function $f \coloneqq \Phi^{-1}(F) \in \V{\B_n}$ is given by $f(x) = \sum_{\{u,v\} \in E} c(\{u,v\}) \cdot f_{u,v}(x)$, where $f_{u,v}(x)=\max\{x_u-x_v,x_v-x_u\}$. To see this, first note that $f \in \V{\B_n}$. Thus, it suffices to check that $F(S)=f(\ind_S)$ for all $S \subseteq [n]$, which follows due to the observation that \[f_{u,v}(\ind_S)=\begin{cases}
    1 & \{u,v\} \in \delta(S) \\
    0 & \{u,v\} \not \in \delta(S)
\end{cases}\]
Hence, \Cref{ex:virtual_zonotopes} implies that the functions \[g= \sum_{ c(\{u,v\}) > 0} c(\{u,v\}) \cdot f_{u,v} \text{ and } h= \sum_{c(\{u,v\}) < 0} c(\{u,v\}) \cdot f_{u,v}\] form the unique minimal decomposition of $f$. Thus, $G=\Phi(g)$ and $H=\Phi(h)$, the submodular functions given by \[G(S) = \sum\limits_{\substack{\{u,v\} \in \delta(S)\\c(\{u,v\}) >0}}c(\{u,v\}) \text{ and } H(S) = \sum\limits_{\substack{\{u,v\} \in \delta(S)\\c(\{u,v\}) <0}}c(\{u,v\})\]  are the unique minimal decompositions of $F$ into submodular functions.
\end{example}

\section{Neural Network Constructions}\label{sec:NN}

In this section we consider the following question: Given a CPWL function $f\colon\R^d\to\R$ with $q$ pieces and~$k$ affine components, what is the necessary depth, width, and size of a neural network exactly representing this function? To this end, we first discuss the necessary background on neural networks and known results on neural complexity. We then prove better results for the case of $f$ being convex. Finally, we extend these results to nonconvex functions by writing them as a difference of convex functions.

\subsection{Background} For a number of hidden layers $\ell\geq0$, a \emph{neural network} with \emph{rectified linear unit} (ReLU) activations is defined by a sequence of $\ell+1$ affine transformations $T_i:\R^{n_{i-1}}\to\R^{n_i}$, $i\in[\ell+1]$. We assume that $n_0=d$ and $n_{\ell+1}=1$. If $\sigma$ denotes the function that computes the ReLU function $x\mapsto \max\{x,0\}$ in each component, the neural network is said to compute the function $f\colon\R^d\to \R$ given by $f=T_{\ell+1}\circ\sigma\circ T_\ell\circ\sigma\circ\dots\circ\sigma \circ T_1$. We say that the neural network has \emph{depth} $\ell+1$, \emph{width} $\max_{i\in[\ell]} n_i$, and \emph{size} $\sum_{i\in[\ell]} n_i$.

It is well-known that the maximum of $d$ numbers can be computed with depth $\lceil \log_2d \rceil+1$ and overall size $\mathcal{O}(d)$ \citep{arora2018understanding}. This simple fact has been used in the literature to deduce exact representations of CPWL functions with neural networks from known representations of CPWL functions. We would like to focus on two of them here, which are in a sense incomparable.
The first one goes back to \citet{hertrich2021towards} and builds upon ideas from \citet{wang2005generalization}. We present it here in a slightly stronger form that follows from \cite{koutschan2023representing}.

\begin{theorem}\label{thm:nnsfromwangsun}
    Every CPWL function $f\colon \R^d\to\R$ with $k$ affine components can be represented by a neural network with depth $\lceil \log_2(d+1) \rceil+1$ and overall size $\mathcal{O}(k^{d+1})$.
\end{theorem}

The second one goes back to \citet{chen2022improved} and is based on the lattice representation of CPWL functions, compare \citet{tarela1999region}.
\begin{theorem}[\citep{chen2022improved}]\label{thm:nnsfromlattice}
    Every CPWL function $f\colon \R^d\to\R$ with $q$ pieces and $k$ affine components can be represented by a neural network with depth $\lceil \log_2k\rceil + \lceil\log_2q\rceil+1$ and overall size $\mathcal{O}(kq)$.   
\end{theorem}
\vspace{-0.1cm}
As noted before, one can assume that $d< k \leq q$, since otherwise we could affinely project to a lower dimension without losing information. In fact, one would usually assume that the input dimension $d$ is much lower than the number of affine components $k$. Therefore, \Cref{thm:nnsfromwangsun} provides the better representation in terms of depth, while \Cref{thm:nnsfromlattice} provides the better representation in terms of size. However, both theorems offer only fixed depth-size tradeoffs: each prescribes a particular depth and size bound, without allowing the user to interpolate between these two regimes. So, the naturally occurring question is: can we somehow freely choose a depth and trade depth against size in these representations? In other words: can we smoothly interpolate between the low-depth high-size representation of \Cref{thm:nnsfromwangsun} and the low-size high-depth representation of \Cref{thm:nnsfromlattice}? In the remaining section we present results that achieve this to some extent.
\subsection{New Constructions for the Convex Case}

In this part, we prove that in the convex case one can obtain a flexible depth-size tradeoff by grouping the affine components, applying the low-depth construction from \Cref{thm:nnsfromwangsun} within each group, and then combining the group maxima using the standard ReLU construction for the maximum function.

\begin{theorem}\label{thm:nnsconvex}
    Every convex CPWL function $f\colon \R^d\to\R$ with $q=k$ affine components can be represented by a neural network with  $\lceil \log_2 (d+1) \rceil +\lceil\log_2 r \rceil$ many hidden layers and overall size $\mathcal{O}(rs^{d+1})$, for any free choice of parameters $r$ and $s$ with $rs\geq k$.
\end{theorem}

\begin{proof}
    Recall that a convex CPWL function can be written as the maximum of its affine components, that is, $f(x)=\max_{i\in[k]}\inner{a_i,x}+b_i$. The idea is to split the $k$ affine components of $f$ into $r$ groups of size at most $s$, apply \Cref{thm:nnsfromwangsun} to compute the maximum within each group, and then simply compute the maximum of the $r$ group maxima in a straight-forward way.

    Let us first focus on computing the maximum of at most $s$ affine components within each of the $r$ groups. By \Cref{thm:nnsfromwangsun}, one can achieve this with a neural network with $\lceil \log_2(d+1) \rceil+1$ many hidden layers and overall size $\mathcal{O}(s^{d+1})$. We put all these $r$ neural networks in parallel to each other and add, at the end, the simple neural network computing the maximum of these $r$ maxima according to \citet{arora2018understanding}, which has $\lceil\log_2 r \rceil$ many hidden layers and overall size $\mathcal{O}(r)$. Altogether, the resulting neural network will have the desired depth and size.    
\end{proof}

In this sense, the construction interpolates between the low-depth representation of \Cref{thm:nnsfromwangsun} and the low-size high-depth behavior of \Cref{thm:nnsfromlattice}.  In order to see how this provides a tradeoff between the representations of \Cref{thm:nnsfromwangsun} and \Cref{thm:nnsfromlattice}, it is worth looking at the extreme cases. If we choose $r=1$ and $s=k$, we exactly obtain the bounds from \Cref{thm:nnsfromwangsun}. On the other hand, if we choose $r=k$ and $s=1$, we obtain a neural network with depth $\mathcal{O}(\log k)$ and size $\mathcal{O}(k)$, which is qualitatively close to the bounds of \Cref{thm:nnsfromlattice}. In fact, the even better size bound stems from the fact that our construction heavily relies on convexity.

In conclusion, by choosing an appropriate $r$ (and corresponding $s$), the user can freely decide how to trade depth against size in neural network representations and thereby interpolate between the two extreme representations of \Cref{thm:nnsfromwangsun} and \Cref{thm:nnsfromlattice}.
\subsection{Extension to the Nonconvex Case}

The construction of \Cref{thm:nnsconvex} provides a nice blueprint of how to interpolate between \Cref{thm:nnsfromwangsun} and \Cref{thm:nnsfromlattice}, but it has the big limitation that it only works in the convex case. Simply mixing the two known representations of \Cref{thm:nnsfromwangsun,thm:nnsfromlattice} does not appear to work in the nonconvex case, as one cannot as easily identify groups of affine components that can be treated separately.

Instead we propose a different approach: given a CPWL function, first split it into a difference of two convex ones and then apply \Cref{thm:nnsconvex} to these two functions. To do this efficiently, it requires to find a good answer to \Cref{problem}.

As discussed in this paper, it is quite challenging to give a satisfying answer to \Cref{problem} in full generality, but there are special cases, where we do have a good answer; see \Cref{sec:compatible}. In these special cases, we can extend the interpolation between the two representations in the following way. We need the following useful lemma.

\begin{lemma}\label{lem:decomposition-exists}
    Let $\complex$ be a regular polyhedral complex. Then every CPWL function compatible with $\complex$ can be written as a difference of two convex CPWL functions that are also compatible with $\complex$. In particular, $\VCmod=\spn(\VCP)$.
\end{lemma}
\begin{proof}
    Let $f\in\VCP$ be an arbitrary function.
    Since $\complex$ is regular, by definition there exists a convex function $g\in\VC$ such that $\complex = \complex_g$. \Cref{lem:nonnegative} implies that 
    $w_g(\sigma) > 0$ for all $\sigma \in \complex^{d-1}$. For sufficiently large $\lambda>0$, it follows that $w_{f+\lambda g}\geq 0$ and thus $f = (f+\lambda g) - \lambda g$ is a representation of $f$ as a difference of two compatible, convex functions, as desired.
\end{proof}

We then obtain the following result by combining \Cref{thm:nnsconvex} with \Cref{lem:decomposition-exists}.

\begin{corollary}\label{cor:nnsregularcomplex}
    Let $f\colon \R^d\to\R$ be a CPWL function that is compatible with a \emph{regular} polyhedral complex $\complex$ with $\tilde{q}=\#\complex^d$ full-dimensional polyhedra. Then, $f$ can be represented by a neural network with $\lceil \log_2(d+1)\rceil + \lceil\log_2 r \rceil$ many hidden layers and overall size $\mathcal{O}(rs^{d+1})$, for any free choice of parameters $r$ and $s$ with $rs\geq \tilde{q}$.
\end{corollary}
\begin{proof}
    By \Cref{lem:decomposition-exists}, $f$ can be decomposed into a difference of two convex functions which are compatible with $\complex$. Consequently, each of them has at most $\tilde{q}$ affine components. Applying \Cref{thm:nnsconvex} to both functions separately and simply putting the two corresponding neural networks in parallel, subtracting the outputs, yields a neural network representing $f$ with the desired size bounds.
\end{proof}

\section{Outlook}\label{sec:openproblems}

From the theoretical perspective, the maybe most dominant open question is the following precise version of \Cref{problem}.

\begin{problem}\label{prob:open}
    Given a CPWL function $f$ in dimension $d$ with $q$ pieces, does there always exist a decomposition $f=g-h$ such that the number of pieces of $g$ and $h$ is polynomial in $d$ and $q$?
\end{problem}

A positive answer to \Cref{prob:open} would have useful consequences for our two applications in the context of (submodular) set function optimization and neural network representations. However, also a negative answer would be equally interesting.

One possible approach to resolve \Cref{prob:open} could be to analyze which objective direction according to \Cref{th:minimizer-unique-vertex} leads to a good vertex of the decomposition polyhedron and prove theoretical properties about that vertex. The same theorem might also be key to developing algorithms that find good decompositions with linear programming. Generally, while beyond the scope of this paper, turning any of our insights into practical algorithms, preferably with theoretical guarantees, is a broad avenue for future research.

\section*{Acknowledgments}
Marie-Charlotte Brandenburg was partially supported by the Wallenberg AI, Autonomous Systems and Software Program (WASP) funded by the Knut and Alice Wallenberg Foundation.
Moritz Grillo was supported by the Deutsche Forschungsgemeinschaft (DFG, German Research
Foundation) under Germany’s Excellence Strategy --- The Berlin Mathematics Research Center
MATH+ (EXC-2046/1, project ID: 390685689).
Part of this work was completed while Christoph Hertrich was affiliated with Université Libre de Bruxelles, Belgium, and received support by the European Union's Horizon Europe research and innovation program under the Marie Skłodowska-Curie grant agreement No 101153187---NeurExCo.

\bibliographystyle{abbrvnat}
\bibliography{ref}

\begin{thebibliography}{53}
\providecommand{\natexlab}[1]{#1}
\providecommand{\url}[1]{\texttt{#1}}
\expandafter\ifx\csname urlstyle\endcsname\relax
  \providecommand{\doi}[1]{doi: #1}\else
  \providecommand{\doi}{doi: \begingroup \urlstyle{rm}\Url}\fi

\bibitem[Aguiar and Ardila(2017)]{Aguiar2017HopfMA}
M.~Aguiar and F.~Ardila.
\newblock Hopf monoids and generalized permutahedra.
\newblock \emph{Memoirs of the American Mathematical Society}, 2017.

\bibitem[Ardila et~al.(2009)Ardila, Benedetti, and Doker]{Ardila2009}
F.~Ardila, C.~Benedetti, and J.~Doker.
\newblock Matroid polytopes and their volumes.
\newblock \emph{Discrete \& Computational Geometry}, 43\penalty0 (4):\penalty0
  841--854, Nov. 2009.
\newblock \doi{10.1007/s00454-009-9232-9}.

\bibitem[Arora et~al.(2018)Arora, Basu, Mianjy, and
  Mukherjee]{arora2018understanding}
R.~Arora, A.~Basu, P.~Mianjy, and A.~Mukherjee.
\newblock Understanding deep neural networks with rectified linear units.
\newblock In \emph{International Conference on Learning Representations}, 2018.

\bibitem[Bakaev et~al.(2025{\natexlab{a}})Bakaev, Brunck, Hertrich, Reichman,
  and Yehudayoff]{bakaev2025depth}
E.~Bakaev, F.~Brunck, C.~Hertrich, D.~Reichman, and A.~Yehudayoff.
\newblock On the depth of monotone {ReLU} neural networks and {ICNNs}.
\newblock \emph{arXiv preprint arXiv:2505.06169}, 2025{\natexlab{a}}.

\bibitem[Bakaev et~al.(2025{\natexlab{b}})Bakaev, Brunck, Hertrich, Stade, and
  Yehudayoff]{bakaev2025better}
E.~Bakaev, F.~Brunck, C.~Hertrich, J.~Stade, and A.~Yehudayoff.
\newblock Better neural network expressivity: subdividing the simplex.
\newblock \emph{arXiv preprint arXiv:2505.14338}, 2025{\natexlab{b}}.

\bibitem[B{\'e}rczi et~al.(2026)B{\'e}rczi, Geh{\'e}r, Imolay, Lov{\'a}sz, and
  Schwarcz]{berczi2024monotonicdecompositionssubmodularset}
K.~B{\'e}rczi, B.~Geh{\'e}r, A.~Imolay, L.~Lov{\'a}sz, and T.~Schwarcz.
\newblock Monotonic decompositions of submodular set functions.
\newblock \emph{SIAM Journal on Discrete Mathematics}, 40\penalty0
  (1):\penalty0 308--335, 2026.

\bibitem[Bertschinger et~al.(2024)Bertschinger, Hertrich, Jungeblut, Miltzow,
  and Weber]{bertschinger2024training}
D.~Bertschinger, C.~Hertrich, P.~Jungeblut, T.~Miltzow, and S.~Weber.
\newblock Training fully connected neural networks is
  $\exists\mathbb{R}$-complete.
\newblock \emph{Advances in Neural Information Processing Systems}, 36, 2024.

\bibitem[Bittner(1970)]{butner70_somerepresentationtheorems}
L.~Bittner.
\newblock Some representation theorems for functions and sets and their
  application to nonlinear programming.
\newblock \emph{Numerische Mathematik}, 16\penalty0 (1):\penalty0 32--51, Feb.
  1970.
\newblock \doi{10.1007/bf02162405}.

\bibitem[Brandenburg et~al.(2024)Brandenburg, Loho, and
  Montufar]{brandenburg2024the}
M.-C. Brandenburg, G.~Loho, and G.~Montufar.
\newblock The real tropical geometry of neural networks for binary
  classification.
\newblock \emph{Transactions on Machine Learning Research}, 2024.

\bibitem[Brandenburg et~al.(2025)Brandenburg, Grillo, and
  Hertrich]{brandenburg2025decomposition}
M.-C. Brandenburg, M.~L. Grillo, and C.~Hertrich.
\newblock Decomposition polyhedra of piecewise linear functions.
\newblock In \emph{The Thirteenth International Conference on Learning
  Representations}, 2025.

\bibitem[Breiman(1993)]{breiman93}
L.~Breiman.
\newblock Hinging hyperplanes for regression, classification, and function
  approximation.
\newblock \emph{IEEE Transactions on Information Theory}, 39\penalty0
  (3):\penalty0 999--1013, 1993.
\newblock \doi{10.1109/18.256506}.

\bibitem[Chen et~al.(2022)Chen, Garudadri, and Rao]{chen2022improved}
K.-L. Chen, H.~Garudadri, and B.~D. Rao.
\newblock Improved bounds on neural complexity for representing piecewise
  linear functions.
\newblock In \emph{Advances in Neural Information Processing Systems},
  volume~35, 2022.

\bibitem[Cybenko(1989)]{cybenko1989approximation}
G.~Cybenko.
\newblock Approximation by superpositions of a sigmoidal function.
\newblock \emph{Mathematics of control, signals and systems}, 2\penalty0
  (4):\penalty0 303--314, 1989.

\bibitem[El~Halabi et~al.(2023)El~Halabi, Orfanides, and Hoheisel]{halabi23}
M.~El~Halabi, G.~Orfanides, and T.~Hoheisel.
\newblock Difference of submodular minimization via {DC} programming.
\newblock In \emph{Proceedings of the 40th International Conference on Machine
  Learning}, volume 202 of \emph{Proceedings of Machine Learning Research},
  pages 9172--9201. PMLR, 23--29 Jul 2023.

\bibitem[Eldan and Shamir(2016)]{eldan2016power}
R.~Eldan and O.~Shamir.
\newblock The power of depth for feedforward neural networks.
\newblock In \emph{Conference on learning theory}, pages 907--940. PMLR, 2016.

\bibitem[Ergen and Grillo(2024)]{ergen2024topological}
E.~Ergen and M.~Grillo.
\newblock Topological expressivity of {ReLU} neural networks.
\newblock In \emph{The Thirty Seventh Annual Conference on Learning Theory},
  pages 1599--1642. PMLR, 2024.

\bibitem[Froese and Hertrich(2024)]{froese2024training}
V.~Froese and C.~Hertrich.
\newblock Training neural networks is {NP}-hard in fixed dimension.
\newblock \emph{Advances in Neural Information Processing Systems}, 36, 2024.

\bibitem[Froese et~al.(2022)Froese, Hertrich, and
  Niedermeier]{froese2022computational}
V.~Froese, C.~Hertrich, and R.~Niedermeier.
\newblock The computational complexity of {ReLU} network training parameterized
  by data dimensionality.
\newblock \emph{Journal of Artificial Intelligence Research}, 74:\penalty0
  1775--1790, 2022.

\bibitem[Froese et~al.(2025)Froese, Grillo, and Skutella]{froese2025complexity}
V.~Froese, M.~Grillo, and M.~Skutella.
\newblock Complexity of injectivity and verification of {ReLU} neural networks.
\newblock In N.~Haghtalab and A.~Moitra, editors, \emph{Proceedings of Thirty
  Eighth Conference on Learning Theory}, volume 291 of \emph{Proceedings of
  Machine Learning Research}, pages 2188--2189. PMLR, 2025.

\bibitem[Grillo et~al.(2025)Grillo, Hertrich, and Loho]{grillo2025depth}
M.~Grillo, C.~Hertrich, and G.~Loho.
\newblock Depth-bounds for neural networks via the braid arrangement.
\newblock In \emph{Conference on Neural Information Processing Systems
  (NeurIPS)}, 2025.

\bibitem[Gr{\"o}tschel et~al.(1981)Gr{\"o}tschel, Lov{\'a}sz, and
  Schrijver]{grotschel1981ellipsoid}
M.~Gr{\"o}tschel, L.~Lov{\'a}sz, and A.~Schrijver.
\newblock The ellipsoid method and its consequences in combinatorial
  optimization.
\newblock \emph{Combinatorica}, 1:\penalty0 169--197, 1981.

\bibitem[Haase et~al.(2023)Haase, Hertrich, and Loho]{haase2023lower}
C.~A. Haase, C.~Hertrich, and G.~Loho.
\newblock Lower bounds on the depth of integral {ReLU} neural networks via
  lattice polytopes.
\newblock In \emph{International Conference on Learning Representations}, 2023.

\bibitem[Hertrich and Loho(2024)]{hertrich2024neural}
C.~Hertrich and G.~Loho.
\newblock Neural networks and (virtual) extended formulations.
\newblock \emph{arXiv preprint arXiv:2411.03006}, 2024.

\bibitem[Hertrich and Sering(2024)]{hertrich2024relu}
C.~Hertrich and L.~Sering.
\newblock {ReLU} neural networks of polynomial size for exact maximum flow
  computation.
\newblock \emph{Mathematical Programming}, pages 1--30, 2024.

\bibitem[Hertrich and Skutella(2023)]{hertrich2023provably}
C.~Hertrich and M.~Skutella.
\newblock Provably good solutions to the knapsack problem via neural networks
  of bounded size.
\newblock \emph{INFORMS journal on computing}, 35\penalty0 (5):\penalty0
  1079--1097, 2023.

\bibitem[Hertrich et~al.(2023)Hertrich, Basu, Di~Summa, and
  Skutella]{hertrich2021towards}
C.~Hertrich, A.~Basu, M.~Di~Summa, and M.~Skutella.
\newblock Towards lower bounds on the depth of {ReLU} neural networks.
\newblock \emph{SIAM Journal on Discrete Mathematics}, 37\penalty0
  (2):\penalty0 997--1029, 2023.

\bibitem[Iyer and Bilmes(2012)]{Iyer2012}
R.~Iyer and J.~Bilmes.
\newblock Algorithms for approximate minimization of the difference between
  submodular functions, with applications.
\newblock In \emph{Proceedings of the Twenty-Eighth Conference on Uncertainty
  in Artificial Intelligence}, UAI'12, page 407–417, 2012.

\bibitem[Jochemko and
  Ravichandran(2022)]{jochemko22_generalizedpermutahedraminkowski}
K.~Jochemko and M.~Ravichandran.
\newblock Generalized permutahedra: {Minkowski} linear functionals and
  {Ehrhart} positivity.
\newblock \emph{Mathematika}, 68\penalty0 (1):\penalty0 217--236, Jan. 2022.

\bibitem[Joswig(2021)]{joswig21_essentialstropicalcombinatorics}
M.~Joswig.
\newblock \emph{Essentials of tropical combinatorics}, volume 219 of
  \emph{Graduate Studies in Mathematics}.
\newblock American Mathematical Society, Providence, RI, 2021.

\bibitem[Koutschan et~al.(2023)Koutschan, Moser, Ponomarchuk, and
  Schicho]{koutschan2023representing}
C.~Koutschan, B.~Moser, A.~Ponomarchuk, and J.~Schicho.
\newblock Representing piecewise linear functions by functions with small
  arity.
\newblock \emph{Applicable Algebra in Engineering, Communication and
  Computing}, pages 1--16, 2023.

\bibitem[Koutschan et~al.(2024)Koutschan, Ponomarchuk, and
  Schicho]{koutschan2024representing}
C.~Koutschan, A.~Ponomarchuk, and J.~Schicho.
\newblock Representing piecewise-linear functions by functions with minimal
  arity.
\newblock \emph{arXiv preprint arXiv:2406.02421}, 2024.

\bibitem[Kripfganz and Schulze(1987)]{KS87}
A.~Kripfganz and R.~Schulze.
\newblock Piecewise affine functions as a difference of two convex functions.
\newblock \emph{Optimization}, 18\penalty0 (1):\penalty0 23–29, 1987.
\newblock \doi{10.1080/02331938708843210}.

\bibitem[Le~Thi and Pham~Dinh(2018)]{le2018dc}
H.~A. Le~Thi and T.~Pham~Dinh.
\newblock {DC} programming and {DCA}: thirty years of developments.
\newblock \emph{Mathematical Programming}, 169\penalty0 (1):\penalty0 5--68,
  2018.

\bibitem[Lov{\'a}sz(1983)]{lovasz1983submodular}
L.~Lov{\'a}sz.
\newblock Submodular functions and convexity.
\newblock \emph{Mathematical Programming The State of the Art: Bonn 1982},
  pages 235--257, 1983.

\bibitem[Maclagan and Sturmfels(2015)]{maclagan2015introduction}
D.~Maclagan and B.~Sturmfels.
\newblock \emph{Introduction to Tropical Geometry}, volume 161 of
  \emph{Graduate Studies in Mathematics}.
\newblock American Mathematical Soc., 2015.

\bibitem[McMullen(1996)]{McMullen1996}
P.~McMullen.
\newblock Weights on polytopes.
\newblock \emph{Discrete \& Computational Geometry}, 15\penalty0 (4):\penalty0
  363--388, Apr. 1996.
\newblock \doi{10.1007/bf02711515}.

\bibitem[Melzer(1986)]{melzer1986expressibility}
D.~Melzer.
\newblock On the expressibility of piecewise-linear continuous functions as the
  difference of two piecewise-linear convex functions.
\newblock \emph{Quasidifferential calculus}, pages 118--134, 1986.

\bibitem[Mont{\'u}far et~al.(2022)Mont{\'u}far, Ren, and
  Zhang]{montufar2022sharp}
G.~Mont{\'u}far, Y.~Ren, and L.~Zhang.
\newblock Sharp bounds for the number of regions of maxout networks and
  vertices of {Minkowski} sums.
\newblock \emph{SIAM Journal on Applied Algebra and Geometry}, 6\penalty0
  (4):\penalty0 618--649, 2022.

\bibitem[Narasimhan and Bilmes(2005)]{Narasimhan2005}
M.~Narasimhan and J.~Bilmes.
\newblock A submodular-supermodular procedure with applications to
  discriminative structure learning.
\newblock In \emph{Proceedings of the Twenty-First Conference on Uncertainty in
  Artificial Intelligence}, UAI'05, page 404–412, 2005.

\bibitem[Ovchinnikov(2002)]{Ovchinnikov2002}
S.~Ovchinnikov.
\newblock Max-min representation of piecewise linear functions.
\newblock \emph{Beitr\"age Algebra Geom.}, 43\penalty0 (1):\penalty0 297--302,
  2002.
\newblock ISSN 0138-4821.

\bibitem[Panina and Streinu(2015)]{panina2015virtual}
G.~Y. Panina and I.~Streinu.
\newblock Virtual polytopes.
\newblock \emph{Russian Mathematical Surveys}, 70\penalty0 (6):\penalty0 1105,
  2015.

\bibitem[Safran et~al.(2024)Safran, Reichman, and Valiant]{safran2024many}
I.~Safran, D.~Reichman, and P.~Valiant.
\newblock How many neurons does it take to approximate the maximum?
\newblock In \emph{Proceedings of the 2024 Annual ACM-SIAM Symposium on
  Discrete Algorithms (SODA)}, pages 3156--3183. SIAM, 2024.

\bibitem[Schl{\"u}ter and Darup(2021)]{schlüter2021novel}
N.~Schl{\"u}ter and M.~S. Darup.
\newblock Novel convex decomposition of piecewise affine functions.
\newblock \emph{arXiv preprint arXiv:2108.03950}, 2021.

\bibitem[Schneider(2013)]{schneider2013convex}
R.~Schneider.
\newblock \emph{Convex bodies: the Brunn--Minkowski theory}, volume 151.
\newblock Cambridge University Press, 2013.

\bibitem[Stargalla et~al.(2025)Stargalla, Hertrich, and
  Reichman]{stargalla2025computational}
M.~Stargalla, C.~Hertrich, and D.~Reichman.
\newblock The computational complexity of counting linear regions in {ReLU}
  neural networks.
\newblock \emph{Advances in Neural Information Processing Systems},
  38:\penalty0 125970--125988, 2025.

\bibitem[Tarela and Martinez(1999)]{tarela1999region}
J.~Tarela and M.~Martinez.
\newblock Region configurations for realizability of lattice piecewise-linear
  models.
\newblock \emph{Mathematical and Computer Modelling}, 30\penalty0
  (11-12):\penalty0 17--27, 1999.

\bibitem[Telgarsky(2016)]{telgarsky2016benefits}
M.~Telgarsky.
\newblock Benefits of depth in neural networks.
\newblock In \emph{Conference on learning theory}, pages 1517--1539. PMLR,
  2016.

\bibitem[Tran and Wang(2024)]{tran2024minimal}
N.~M. Tran and J.~Wang.
\newblock Minimal representations of tropical rational functions.
\newblock \emph{Algebraic Statistics}, 15\penalty0 (1):\penalty0 27--59, 2024.

\bibitem[Wang(2004)]{1333237}
S.~Wang.
\newblock General constructive representations for continuous piecewise-linear
  functions.
\newblock \emph{IEEE Transactions on Circuits and Systems I: Regular Papers},
  51\penalty0 (9):\penalty0 1889--1896, 2004.
\newblock \doi{10.1109/TCSI.2004.834521}.

\bibitem[Wang and Sun(2005)]{wang2005generalization}
S.~Wang and X.~Sun.
\newblock Generalization of hinging hyperplanes.
\newblock \emph{IEEE Transactions on Information Theory}, 51\penalty0
  (12):\penalty0 4425--4431, 2005.

\bibitem[Zalgaller(2000)]{Zalgaller2000}
V.~A. Zalgaller.
\newblock Representation of functions of several variables by differences of
  convex functions.
\newblock \emph{Journal of Mathematical Sciences}, 100\penalty0 (3):\penalty0
  2209--2227, June 2000.
\newblock ISSN 1573-8795.
\newblock \doi{10.1007/s10958-000-0006-4}.

\bibitem[Zhang et~al.(2018)Zhang, Naitzat, and Lim]{zhang2018tropical}
L.~Zhang, G.~Naitzat, and L.-H. Lim.
\newblock Tropical geometry of deep neural networks.
\newblock In \emph{International Conference on Machine Learning}, pages
  5824--5832. PMLR, 2018.

\bibitem[Ziegler(2012)]{ziegler_lecturespolytopes}
G.~M. Ziegler.
\newblock \emph{Lectures on Polytopes}.
\newblock Springer New York, May 2012.
\newblock ISBN 038794365X.
\newblock URL
  \url{https://www.ebook.de/de/product/3716808/guenter_m_ziegler_lectures_on_polytopes.html}.

\end{thebibliography}

\end{document}